\RequirePackage{fix-cm}
\documentclass[11pt]{amsart}
\usepackage{graphicx, amssymb, amsmath,amsthm}

\usepackage{enumerate}
\usepackage{graphicx}
\usepackage{geometry}
\usepackage{mathtools}
\usepackage{color}
\usepackage[mathscr]{euscript}
\usepackage{amssymb}
\usepackage{bbm}
\usepackage[colorlinks,
            linkcolor=blue,
            anchorcolor=blue,
            citecolor=blue]{hyperref}

\newtheorem{thm}{Theorem}

\newcommand{\ds}{\displaystyle}

\def \a{\alpha} \def \b{\beta} \def \g{\gamma} \def \d{\delta}
   
\def \s{\sigma} \def \l{\lambda} \def \z{\zeta} \def \o{\omega}
\def \O{\Omega} \def \D{\Delta}  \def\P{\mathcal{P}}
 \def \r{\rho}
\def \k{\kappa}  \def \G{\Gamma}

\def \L{\Lambda}
\def\R{\mathbb{R}}
\def\To{\mathbb{T}}
\def \ra{\rightarrow}
 \def\N{\mathcal{N}}
\newtheorem{theorem}{Theorem}[section]

\newtheorem{lemma}[theorem]{Lemma}
\newtheorem{proposition}[theorem]{Proposition}

\newtheorem{definition}[theorem]{Definition} 

\numberwithin{equation}{section}

\begin{document}


 \author{Rongchang Liu} \address[Rongchang Liu] {  Department of Mathematics\\
University of Arizona\\Tucson, AZ 85721, USA}
 \email[R.~Liu]{lrc666@math.arizona.edu}

\author{Kening Lu} \address[Kening Lu] {  School of Mathematics\\
Sichuan University\\
Chengdu, Sichuan 610064, PR China}
\email[k.~Lu]{keninglu@scu.edu.cn}

\title[Exponential mixing and limit theorems of quasi-periodically forced 2D stochastic Navier-Stokes Equations in the hypoelliptic setting]{Exponential mixing and limit theorems of quasi-periodically forced 2D stochastic Navier-Stokes Equations in the hypoelliptic setting}

\pagestyle{plain}

\begin{abstract} {
We consider the incompressible 2D Navier-Stokes equations on the torus driven by a deterministic time quasi-periodic force and a noise that is white in time and degenerate in Fourier space. We show that the asymptotic statistical behavior is characterized by a quasi-periodic invariant measure that exponentially attracts the law of all solutions. The result is true for any value of the viscosity $\nu>0$ and does not depend on the strength of the external forces.

By utilizing this quasi-periodic invariant measure, we establish a quantitative version of the strong law of large numbers and central limit theorem for the continuous time inhomogeneous solution processes with explicit convergence rates. It turns out that the convergence rate in the central limit theorem depends on the time inhomogeneity through the Diophantine approximation property on the quasi-periodic frequency of the quasi-periodic force.

} \end{abstract}

\maketitle

\baselineskip 14pt

\tableofcontents
\section{Introduction}

We study the asymptotic statistical properties of time inhomogeneous solution processes of the incompressible  2D  Navier-Stokes equations driven by a deterministic time quasi-periodic force and a stochastic force that is white in time and degenerate in Fourier space. The equation is studied over the two dimensional torus $\To^2: =\mathbb{R}^2/(2\pi)\mathbb{Z}^2$, which in the vorticity form reads
\begin{align}\label{NS-intro}
dw(t, x) + B(\mathcal{K} w, w) (t, x)d t = \nu\mathrm{\mathrm{\Delta}} w(t, x)dt  + f(t, x)dt + \sum_{i=1}^d g_idW_i(t), \quad t>s, \quad w(s) = w_0,
\end{align}
where $w(t, x)$ is the vorticity field, and $\mathcal{K} w$ is the divergence free velocity field. The phase space is chosen as $H:= \left\{w \in \mathrm{L}^2\left(\mathbb{T}^{2}, \mathbb{R}\right): \int_{\mathbb{T}^2} w d x=0\right\}$ whose norm is denoted by $\|\cdot\|$ and the inner product is $\langle\cdot,\cdot\rangle$. We also define the interpolation spaces $H_{s}=\left\{w \in H^{s}\left(\mathbb{T}^{2}, \mathbb{R}\right): \int_{\To^2} w d x=0\right\}$ and the corresponding norms $\left\|\cdot\right\|_s$ by $\|w\|_{s}=\left\|\left(-\mathrm{\Delta}\right)^{s/2} w\right\|$.
The deterministic force  $f$ is quasi-periodic in $t$ with a frequency vector $\alpha = (\alpha_1, \alpha_2, \cdots, \alpha_n)$ and $\{\alpha_k\}_{k=1}^n$ are rationally independent. $W = (W_1, W_2, \cdots, W_d)$ is a two-sided $\R^d$-valued standard Wiener process over the sample space $(\O, \mathcal{F}, \mathbf{P})$ where $\mathbf{P}$ is the Wiener measure,  and $\{g_i\}$ are smooth elements in $H$. Under appropriate spatial regularity conditions on the external forces, the equation is well posed with a time inhomogeneous Markov solution process $w_{s, t}(w_0)$. It generates a two-parameter Markov transition operator $\P_{s, t}$
acting on the space of bounded measurable functions ${B}_b(H)$ as 
\begin{align}\label{intro-markov-ns}
	\P_{s, t}\phi(w_0) = \mathbf{E}\phi(w_{s, t}(w_0)), \quad\forall \phi\in B_b(H), w_0\in H. 
\end{align}
It acts on the space of probability measures $\mathcal{P}(H)$ by duality 
\begin{align}\label{intro-markov-ns-dual}
\mathcal{P}_{s, t}^*\mu (A) = \int_{H} \mathcal{P}_{s, t}\mathbb{I}_{A}(w)\mu(dw), \text{ for } \mu\in \mathcal{P}(H), A\in \mathfrak{B},
\end{align}
where $\mathfrak{B}$ is the Borel $\s$-algebra of $H$ and $\mathbb{I}_{A}$ is the indicator function of $A$. For $\eta>0$, recall the metric $\r$ weighted by a Lyapunov function introduced by Hairer and Mattingly \cite{HM08}, 
\begin{align}\label{rho-introduction}
\rho(w_1, w_2) = \inf_{\gamma}\int_0^1e^{\eta\|\gamma(t)\|^2}\|\dot{\gamma}(t)\|dt, \quad\forall w_1, w_2 \in H,
\end{align}
where the infimum is taken over all differentiable path $\gamma$ connecting $w_1, w_2\in H$, and $\dot{\gamma}$ represents the time derivative. We endow $\P(H)$ with the topology of weak convergence and denote by $\P_1(H)$ the set of probability measures that have finite first moment with respect to the 1-Wasserstein metric induced from the metric $\r$ in $H$. A quasi-periodic invariant measure is a continuous quasi-periodic map with values in $\P(H)$ satisfying the following invariance condition 
\begin{align}\label{introduction-invariance}
\mathcal{P}_{s,t}^*\mu_s = \mu_t, \quad s \leq t.
\end{align}
Note that the concept of quasi-periodic invariant measure for SDEs was first introduced by Feng, Qu and Zhao \cite{FQZ21}, where they do not place the continuity condition. 

\vskip0.05in

Since we are working with a degenerate noise, we recall the Lie brackets from Hairer and Mattingly \cite{HM11}. Define the set $A_{\infty}$ by setting
\begin{align}\label{A-infty}
	A_1=\{g_l : 1\leq l \leq d\}, A_{k+1} = A_{k}\cup \{\widetilde{B}(u,v): u, v\in A_{k}\}, A_{\infty} = \overline{\mathrm{span}(\cup_{k\geq 1} A_{k})},
\end{align}
 where $\widetilde{B}(u,w)=-B(\mathcal{K}u,w)-B(\mathcal{K}w,u)$ is  the symmetrized nonlinear term. Besides regularity conditions on the external forces,  the only remaining assumption for our main results is the Hörmander condition: $A_{\infty}=H$. The noise is allowed to be activated through only four modes to ensure the Hörmander condition, see Hairer and Mattingly \cite{HM06}.

\vskip0.1in
\noindent
\textbf{Exponentially mixing quasi-periodic invariant measure}
\vskip0.1in

The celebrated works of Doeblin \cite{Doe37} and Harris \cite{Har56} reveal that the steady state of a time homogeneous Markov system can be described by an invariant measure that is exponentially mixing.  For time inhomogeneous Navier-Stokes equations, the only work we know that addressed exponential mixing for the time inhomogeneous solution processes is Da Prato and Debussche \cite{DD08}, where a non-degenerate noise and a time-periodic deterministic force are considered, which is different from the current setting. 

\vskip0.05in

Our Theorem \ref{theorem-A} below shows that for the quasi-periodically forced 2D Navier-Stokes system with a degenerate noise, there is a quasi-periodic invariant measure that exponentially attracts the law of all solutions. 

\begin{thm}[Exponential Mixing]\label{theorem-A}
 Assume $A_{\infty}=H$ and $f\in C_b(\R, H_2)$ is quasi-periodic. There exists  a quasi-periodic path $\{\mu_t\}_{t\in\R}$ in $\P(H)$ satisfying  the invariance condition \eqref{introduction-invariance}, such that
\begin{align}\label{Intro-ergodic-mixing}
\r(\P_{s, t}^*\mu, \mu_{t})\leq C e^{-\varpi(t-s)}\r(\mu, \mu_s), \quad\forall s\leq t, \mu\in\P(H),
\end{align}
where $C, \varpi$ are positive constants. Moreover, if $f$ is H\"older continuous in time, then the quasi-periodic path is also H\"older continuous. 
\end{thm}

The proof of Theorem \ref{theorem-A} will be given in section \ref{section-ergodicmixing} by a fixed point argument to the induced action on the space of quasi-periodic graphs $C(\To^n, \P_1(H))$. More precisely, we first employ the idea of time symbols as in Chepyzhov and Vishik \cite{CV02a},  and consider a family of transition operators $\P_{s, t, h}$  indexed by time symbols $h\in\To^n$ such that $\P_{s, t, 0}=\P_{s, t}$ as in \eqref{intro-markov-ns}.  This idea captures the propagation of time inhomogeneity and its impact on the dynamics. The action of the dual $\P_{s, t, h}^*$ on $\P(H)$ then becomes a cocycle over the irrational rotation flow $\b_th=h+\a t$ on $\To^n$. Then we utilize the ``pullback" idea from the theory of random dynamical systems, to lift the cocycle to a semigroup acting on the space of quasi-periodic graphs $C(\To^n, \P_1(H))$. The fixed point $\G$ of this semigroup, when evaluated along the specific irrational rotation trajectory $\{\b_t0\}$, gives the desired quasi-periodic invariant measure $\mu_t:=\G(\a t)$ of our original system \eqref{NS-intro}.  The contraction needed in the fixed point argument is proved by adapting the Harris-like theorem  in \cite{HM08} to our time inhomogeneous setting. 

\vskip0.05in

A byproduct of the strategy is the Hölder regularity in $h\in\To^n$ of the fixed point,  which plays a crucial role in analyzing the impact of time inhomogeneity on the convergence rate of the central limit theorem through a Diophantine condition on the frequency vector $\a$. 

\vskip0.1in
\noindent
\textbf{Quantitative limit theorems with convergence rates}
\vskip0.1in
In the time homogeneous and essentially elliptic setting, the estimates of the rate of convergence of the limit theorems (strong law of large numbers and central limit theorem) were obtained for the 2D stochastic Navier-Stokes equation by Shirikyan \cite{Shi06}. To the best of our knowledge, there is no such work for continuous time inhomogeneous Markov processes.  

Our Theorem \ref{theorem-B} below gives quantitative results on the limit theorems with explicit convergence rates, for the time inhomogeneous solution processes of the Navier-Stokes system \eqref{NS-intro}. 

\begin{thm}[Quantitative Limit Theorems]\label{theorem-B} Assume $A_{\infty}=H$ and $f\in C_b(\R, H_2)$ is quasi-periodic.
\vskip0.05in
1. Strong law of large numbers with convergence rate: 
for any $s\in\R$ and $\varepsilon>0$, there is an almost surely finite random time $T_{s, \varepsilon}$ such that for all $T\geq T_{s, \varepsilon}$, 
\begin{align}\label{Intro-rate-of-convergence-SLLN}
\left|\frac1T\int_{0}^T\Big(\phi(w_{s, s+t}(w_0)) - \langle\mu_{s+t}, \phi\rangle\Big) dt\right|\leq T^{-\frac12 +\varepsilon}, 
\end{align}
\vskip0.05in
2. If we further assume that $f$ is H\"older continuous in time and its frequency vector satisfies a Diophantine condition, then we have the central limit theorem with convergence rate:  
\begin{align}\label{Intro-rate-of-convergence-CLT}
\sup _{z \in \mathbb{R}}\left(\xi_{\sigma}(z)\left|\mathbf{P}\left\{\frac{1}{\sqrt{T}}\int_0^T\Big(\phi\left(w_{s, s+t}(w_0)\right)- \left\langle\mu_{s+t},\phi \right\rangle\Big) dt \leq z\right\}-\N_{\sigma}(z)\right|\right) \leq C_{\varepsilon_0} T^{-\frac{1}{4}+\varepsilon_0},
\end{align}
where $\xi_{\sigma} \equiv 1$ for $\sigma>0, \xi_{0}(z)=1 \wedge|z|$, and $\N_{\sigma}(z)$ is the distribution function of the centered Gaussian distribution with variance $\s^2$. 
\end{thm}

Here $\varepsilon_0\in (0, \frac14)$ is a constant depending on  the mixing rate of the quasi-periodic invariant measure and  the Diophantine condition on the frequency, which cannot be arbitrarily small in contrast to the time homogeneous situation.

\vskip0.05in

Theorem \ref{theorem-B} is proved in section \ref{sectionlimittheorem}. 
The general idea is to derive first the estimates for the approximating martingale and then pass to inequalities \eqref{Intro-rate-of-convergence-SLLN} and \eqref{Intro-rate-of-convergence-CLT} by establishing a particular martingale approximation scheme that is valid for the inhomogeneous processes.  In particular, estimate \eqref{Intro-rate-of-convergence-CLT} is derived from a combination of several ideas from \cite{Shi06} with a Berry-Esseen type result for martingales from Hall and Heyde \cite{HH14}, and a convergence rate of the Birkhoff sums for the irrational rotation established by Klein, Liu and Melo \cite{KLM19}. 

\vskip0.05in

There are three main features that are different from the time homogeneous case. First, since the irrational rotation hiding in the time inhomogeneity is not mixing, the usual martingale approximation cannot be directly applied to the corresponding homogenized process. We eliminate the spectral projection of the observable function on the quasi-periodic invariant measure to obtain a valid martingale approximation. Secondly, to analyze the impact of time inhomogeneity on the convergence rate, we perform a detailed analysis on the Hölder regularity of a particular induced observable function on the torus $\To^n$ involving the quasi-periodic invariant measure. Thirdly, due to the interaction between the mixing rate of the solution process and the irrational rotation (related to the Diophantine condition) inherited from the quasi-periodic force, the convergence rate in our context cannot be arbitrarily close to the likely optimal rate obtained in the time homogeneous case \cite{Shi06}. 

\vskip0.05in

We end the introduction by a brief review on related literature in fluid dynamics. The study of Navier-Stokes equations with time-periodic forces dates back to  Serrin \cite{Ser59}, Yudovich \cite{Yud60}, Prouse \cite{Pro63a}, Lions \cite{Lio69} and  many others, see Galdi and Kyed \cite{GK18} for recent progress. The case of quasi-periodic forces was first investigated by Prouse \cite{Pro63b}, where the existence of weak quasi-periodic solutions in dimension two was obtained. Later Ruelle \cite{Rue84} studied dissipative systems driven by quasi-periodic forces including the Navier-Stokes system, for which he gave upper bounds for the entropy and dimension of the attractor in terms of the Grashof number. Attractors of dissipative equations driven by various time dependent forces have been systematically studied by Vishik with his coauthors and many others, see \cite{CV02a, CV08} and references therein. 
Quasi-periodic motions in PDEs have also been widely studied over the years through the KAM (Kolmogorov-Arnold-Moser) 
theory, since the pioneering works of Kuksin \cite{Kuk87,EK09}, Wayne \cite{Way90}, and Bourgain \cite{Bou98}, etc.  We refer the readers to Berti \cite{Ber16} for a survey on this topic.  By employing KAM techniques, Baldi and Montalto \cite{BM21} recently  constructed quasi-periodic solutions to the 3D Euler equations subject to time quasi-periodic forces. Franzoi and  Montalto \cite{FM22} also addressed the quasi-periodic solutions  of the 2D Navier-Stokes equations with a time quasi-periodic external force. 

\vskip0.05in

Our results  give a statistical description for trajectories of the quasi-periodically forced Navier-Stokes system with a degenerate noise perturbation.
The validity of the results for all $\nu>0$ covers the turbulent regime $\nu\ll 1$ of broad interest, which might be useful in the study of inviscid limit ($\nu\to0$) problems.

\section{Settings and main results}\label{setting}
Throughout the paper we will use  $\langle \mu, \phi\rangle$ or $\langle  \phi, \mu\rangle$ to represent the integration of  a real valued function $\phi$ with respect to a given measure $\mu$. The constant $C$ could be different from line to line, but we will emphasize its dependence on the parameters when necessary. 

\subsection{Quasi-periodic invariant measures}
Let $(M, d)$ be a metric space with metric $d$ and $C_b(\mathbb{R}, M)$ the space of bounded continuous functions endowed with the uniform convergence topology generated by the following metric
\[\underline{d}(q_1, q_2) = \sup_{t\in\mathbb{R}}d(q_1(t), q_{2}(t)).\]
\begin{definition}[Quasi-periodic functions \cite{CV02a}]\label{def-quasi-periodic-function}
 A function $q \in C_b(\mathbb{R}, M)$ is quasi-periodic with frequency $\alpha = (\alpha_1, \alpha_2,\cdots, \alpha_n)\in \mathbb{R}^n$  if there is $Q\in C(\mathbb{T}^n, M)$ such that
\begin{align}\label{quasi-periodic-function}
q(t) = Q(\alpha t) = Q(\alpha_1t,\alpha_2t,\cdots,\alpha_nt),
\end{align}
and $\alpha_1, \alpha_2,\cdots, \alpha_n$ are rationally independent, where $\mathbb{T}^n = \mathbb{R}^n/(2\pi)\mathbb{Z}^n$ is the $n$-dimensional torus. 
\end{definition}
We remark here that the H\"older continuity of $q$ is equivalent to that of $Q$ since any trajectory of the irrational rotation flow is dense. We also need the notion of a Diophantine condition on the frequency  $\alpha$ as in \cite{KLM19}. 
\begin{definition}[Diophantine condition]\label{Diophantine condition}
A frequency   $\alpha \in \To^n$ is said to satisfy a Diophantine condition if there exist $K>0$ and $A>n$ such that 
\begin{align}\label{eq-Diophantine condition}
\mathrm{dist}(k\cdot\alpha, 2\pi\mathbb{Z})\geq \frac{K}{\|k\|^{A}},
\end{align}
for all $k\in\mathbb{Z}^n$ with $\|k\|\neq 0$, where $\ds\|k\|: = \max_{1\leq i\leq n}|k_i|$, and $k\cdot\alpha = k_i\alpha_1+k_2\alpha_2+\cdots+k_n\alpha_n$. 

\vskip0.05in


\end{definition}

\vskip0.05in

The equation under investigation is 
\begin{align}\label{NS}
dw(t, x) + B(\mathcal{K} w, w) (t, x)d t = \nu\mathrm{\mathrm{\Delta}} w(t, x)dt  + f(t, x)dt + GdW(t), \quad t>s, \quad w(s) = w_0,
\end{align}
where $s \in \mathbb{R}$, $w(t, x)$ is the scalar vorticity field,  $\mathrm{\mathrm{\Delta}}$ is the Laplacian operator with periodic boundary conditions, $B(\mathcal{K}w, w)= (\mathcal{K}w)\cdot\nabla w$ is the nonlinear term and $\mathcal{K}$ is Biot-Savart integral operator.  The deterministic initial condition $w_0$ lies in  the state space $H$ of system \eqref{NS},  whose definition as well as related Sobolev spaces $H_s$ are as given in the introduction. 

\vskip0.05in

Here $W(t) $ is a standard $d$-dimensional two-sided  Brownian motion obtained as follows.  Let $W^{\pm}(t)$ be two independent standard $d$-dimensional Brownian motion, then  define
\begin{align*}
W(t) :=\left\{\begin{array}{rr}
W^{+}(t) \text{ , } t\geq 0,\\
W^{-}(-t) \text{ , } t < 0 .\\
\end{array}
\right.
\end{align*}
The sample space is denoted by  $(\mathrm{\mathrm{\Omega}} ,\mathcal{F},\mathbf{P})$, where $\mathrm{\Omega} = \left\{\omega\in C(\mathbb{R}, \mathbb{R}^d): \omega (0)=0\right\}$ is endowed with the compact open topology, $\mathcal{F}$ is the Borel $\sigma$-algebra and $\mathbf{P}$  is the Wiener measure associated with the Brownian motion $W$. Denote by $\mathcal{F}_t: = \s(W(u)-W(v): -\infty<u, v\leq t)$ the filtration of $\sigma$-algebras  generated by  $W(t)$.  The coefficient of the noise is a bounded linear operator $G: \mathbb{R}^d\rightarrow H_{\infty}:= \bigcap_{s>0} H_{s}$, such that $Ge_{i}= g_i$ , where $\{e_i\}$ is the standard basis of $\mathbb{R}^d$. Then the noise can be expressed as $GW(t) = \sum_{i=1}^{d} g_{i}W_{i}(t)$. 
\vskip0.05in

 We assume that the deterministic force $f\in C_b\left(\mathbb{R}, H_2\right)$ is time quasi-periodic with frequency $\alpha = (\alpha_1, \alpha_2,\cdots, \alpha_n)$. Then there is a function $\Psi\in  C(\mathbb{T}^n, H_2)$ such that $f(t, x) = \Psi(\a t, x)$. 
 
 \vskip0.05in

 The existence and uniqueness of the solution to equation \eqref{NS}  is well known \cite{Fla94,KS12}. Namely, under the above conditions on the forces, for any initial time $s\in \mathbb{R}$ and $w_0\in H$, equation \eqref{NS} has a unique solution whose sample paths belong to $C\left([s,\infty); H\right)\cap L_{\mathrm{loc}}^2\left((s, \infty);H_3\right)$ almost surely, generating a stochastic flow. Throughout this work, we will use $w_{s, t}(\omega, w_0)$ or $\Phi_{s, t}(\omega, w_0)$ to represent the solution with initial data $w_0$ at initial time $s$. 

\vskip0.05in

The  Markov transition operators $\mathcal{P}_{s, t}$ generated by solutions to \eqref{NS} and the corresponding dual operators $\mathcal{P}_{s, t}^*$ are given as in \eqref{intro-markov-ns} and \eqref{intro-markov-ns-dual}. For $\eta>0$, the metric $\rho$ on $H$ weighted by the Lyapunov function $e^{\eta\|w\|^2}$ as in \eqref{rho-introduction} induces a Wasserstein metric on $\mathcal{P}(H)$ by
\begin{align}\label{rho-on-PH}
\rho(\mu_1, \mu_2) = \inf_{\mu\in\mathcal{C}(\mu_1, \mu_2)}\int_{H\times H} \rho(u, v)\mu(dudv),
\end{align}
where $\mathcal{C}(\mu_1, \mu_2)$ is the set of couplings of $\mu_1, \mu_2\in \mathcal{P}(H)$.  The subset
\begin{align}\label{P1H}
\mathcal{P}_1(H) : = \left\{\mu\in\mathcal{P}(H): \rho(\mu,\delta_0)<\infty\right\}
\end{align}
 is complete under the metric $\rho$ where $\delta_0$ is the Dirac measure at $0$.
 For this Wasserstein metric, the following Monge-Kantorovich duality is well-known \cite{Ch04,Vil08},
 \begin{align}\label{Monge-Kantorovich duality}
 \r\left(\mu_{1}, \mu_{2}\right)=\sup _{\mathrm{Lip}_{\r}(\phi) \leq 1}\left|\int \phi(x) \mu_{1}(d x)-\int \phi(x) \mu_{2}(d x)\right|, \quad\forall \mu_1,\mu_2\in\P_1(H), 
 \end{align}
 where $\mathrm{Lip}_{\r}(\phi)$ is the Lipschitz constant of the function $\phi$ on $H$ endowed with the metric $\r$. 
\begin{definition}[Quasi-periodic invariant measures]
A quasi-periodic invariant measure of system \eqref{NS} is a quasi-periodic function $\mu\in C(\R, \P(H))$ that is invariant under the Markov transition operators
\begin{align*}
\int_{H} \mathcal{P}_{s, t} \varphi (w) \mu_{s}(d w)=\int_{H} \varphi(w) \mu_{t}(d w), \quad s \leq t,\quad \varphi \in B_{b}(H),
\end{align*}
or equivalently
\begin{align*}
\mathcal{P}_{s,t}^*\mu_s = \mu_t, \quad s \leq t.
\end{align*}
\end{definition}


\subsection{Time symbols formulation}\label{skew-product-homogenization}
To deal with the time inhomogeneity, we take a classical method that has been widely used in the study of non-autonomous problems arising from deterministic differential equations and dynamical systems  \cite{CV02a}.

%
%
%

\vskip0.05in

Let $\b_th: =h+\alpha t$ be the irrational rotation flow on $\To^n$ with frequency vector $\a$ that is the same as the frequency of $f$. Consider a family of Navier-Stokes equations indexed by time symbols $h\in\To^n$ which is obtained from \eqref{NS} by replacing $f(t, x)$ with $\Psi(\b_th, x)$,
\begin{align}\label{Settings-NS-family}
dw(t, x) + B(\mathcal{K} w, w) (t, x)d t = \nu\mathrm{\mathrm{\Delta}} w(t, x)dt  + \Psi(\b_th, x)dt + GdW(t), \quad t>s, \quad w(s) = w_0,
\end{align}
with corresponding solution $w_{s, t, h}(w_0)$ (also denoted as $\Phi_{s, t, h}(w_0)$) and transition operators $\P_{s, t, h}$ defined in the same way as \eqref{intro-markov-ns}.  In this vein, the original system \eqref{NS} is embedded into this family of systems and $\P_{s, t}=\P_{s, t, 0}$ since $f(t, x) = \Psi(\b_t0, x)$.  Moreover, it follows from the uniqueness of solution for \eqref{Settings-NS-family} that the following translation identity holds:
\begin{align}\label{translation-identity}
\P_{s+\tau, t+\tau, h} = \P_{s, t, \b_{\tau}h}, \, s\leq t, \tau\in\R, \, h\in \To^n.
\end{align}
Hence by the evolution property of the transition operators, the family is a random dynamical system over the irrational rotational flow, 
\begin{align}\label{cocycle-property-dual}
	\P_{0, t+s, h} = \P_{0, s, h}\P_{0, t, \b_sh} , \, s, t \geq0, h\in\To^n.
\end{align}

\vskip0.05in

The following coupled process on $H\times\To^n$ 
\begin{align}\label{homogenized-process}
	X_t(w, h) = (\Phi_{0, t, h}(w), \b_th), \, (w, h)\in H\times\To^n,
\end{align}
is the associated time homogenous process. We denote its Markov semigroup as $P_t$.

\subsection{Main results}

In this subsection, we formulate the main results of the present paper in details.  
We use the metric \eqref{rho-on-PH} on $\mathcal{P}(H)$ to measure the convergence to the quasi-periodic invariant measure. Note that the metric $\rho$ depends on the parameter $\eta>0$.

\vskip0.05in

The following Theorems \ref{ergodicmixing}-\ref{rate-of-convergence-CLT} are our main results under the standing assumption: 
\begin{align}\label{assumption}
f\in C_b(\R, H_{2}) \text{ is quasi-periodic};  g_i\in H_{\infty}, \forall 1\leq i\leq d;  \text{ and } A_{\infty} = H,
\end{align}
where $A_{\infty}$ is defined in the introduction \eqref{A-infty}.

\vskip0.05in

The first result that will be proved in Section \ref{section-ergodicmixing} is the following exponential mixing of the quasi-periodic invariant measure for \eqref{NS} under the Wasserstein metric \eqref{rho-on-PH}.
\begin{theorem}[Exponential mixing]\label{ergodicmixing}
Assume the standing assumption \eqref{assumption}. There is a unique quasi-periodic invariant measure $\mu_t$ for \eqref{NS} given by a unique map $\G\in C(\To^n, \mathcal{P}(H))$, i.e., 
\[\mu_t= \G_{\b_t 0},  \text{ and } \mathcal{P}_{s, t}^{*}\mu_{s} = \mu_{t}.\]
Moreover, there exists $\eta_0>0$, such that for every $\eta\in (0, \eta_0]$, there are constants $C, \varpi>0$, such that $\G\in C(\To^n, \mathcal{P}_1(H))$ and 
\begin{align}\label{mixing-PH-f}
\rho(\mathcal{P}_{s, s + t}^{*}\mu, \mu_{s+t})\leq Ce^{-\varpi t}\rho(\mu, \mu_{s}),\quad \forall s\in\mathbb{R},  t\geq 0, \mu\in\mathcal{P}(H),
\end{align}
where $C, \varpi$ do not depend on $s$. Furthermore, $\G\in C^{\z}(\To^n, (\P_1(H),\r))$ if $\Psi\in C^{\g}(\To^n, H)$, where $\z = \frac{\varpi \g}{r_0+\varpi}$ with $r_0$ from Lemma \ref{bounds} in the Appendix. 
\end{theorem}

\vskip0.05in

The second result is on the quantitative strong law of large numbers  (SLLN) and the central limit theorem (CLT) for the time inhomogeneous solution processes. The proof will be given in Section \ref{sectionlimittheorem}. To state the results, we first define the space of observable functions.
 For $\g\in (0, 1]$, let  $C^{\g}_{\eta}(H)$ be the space of Hölder continuous functions with finite norms weighted by the Lyapunov function $e^{\eta\|w\|^2}$, 
\begin{align}\label{weighted-holder-space-not-depend-on-torus}
C^{\g}_{\eta}(H): =\left\{ \phi: H\ra\R : \|\phi\|_{\g,\eta}<\infty\right\},
\end{align}
where 
\begin{align*}
\|\phi\|_{\g,\eta}:=\sup_{w\in H}\frac{|\phi(w)|}{e^{\eta\|w\|^2}} + \sup_{0<\|w_1-w_2\|\leq1}\frac{|\phi(w_1)-\phi(w_2)|}{\|w_1-w_2\|^{\g}\left(e^{\eta\|w_1\|^2}+e^{\eta\|w_2\|^2}\right)}.
\end{align*}
Recall that  $\Phi_{s, s+t}(w)$ is the solution to \eqref{NS} starting from $w\in H$ at time $s\in\R$.

\vskip0.05in

\begin{theorem}[Quantitative limit theorems]\label{rate-of-convergence-CLT}
Assume the standing assumption \eqref{assumption}. There is a constant $\eta_0>0$ such that the following estimates hold. 

\vskip0.05in

1. (SLLN with convergence rate) Let $\varepsilon>0$, for every integer $p\geq 3$ satisfying $2^p>1/\varepsilon$, every $\eta\in (0, 2^{-p-1}\eta_0]$, and every $\phi\in C^{\g}_{\eta}(H)$, $w\in H$, $s\in\R$, there is an almost surely finite random time $T_0(\o) \geq 1$, depending on $p, \varepsilon, \|\phi\|_{\g,\eta, H}, s, \|w\|$ such that for all $T>T_0$,  we have 
\begin{align*}
\left|\frac{1}{T}\int_0^T\Big(\phi(\Phi_{s, s+t}(w))- \big\langle \mu_{s+t}, \phi\big\rangle\Big)dt\right|\leq T^{-\frac12+\varepsilon}, 
\end{align*}
Moreover, for every $0<\ell<\min\{2^p\varepsilon-1, 2^{p-2}-1\}$, there is a constant $C_p = C_p(\|\phi\|_{\g,\eta, H}, \ell, \varepsilon)$ such that 
\[\mathbf{E}T_0^{\ell}\leq C_pe^{2^{p+1}\eta\|w\|^2}. \]

\vskip0.05in

2. (CLT with convergence rate) Assume $f\in C^{\g}(\R, H_2)$ and the frequency $\alpha$ satisfies the Diophantine condition  \eqref{Diophantine condition} with constant $A$ and dimension $n$. Let $\N_{\sigma}$ be the distribution function of the normal random variable $N(0, \s^2)$. Let $\overline{\g}_0 =\z^3/125$, where $\z$ is the H\"older exponent of $\G$ from Theorem \ref{ergodicmixing}. 

\vskip0.05in

(1). For any $\eta\in(0,\eta_0/16]$, and every $\phi\in C^{\g}_{\eta}(H)$, $w\in H$, $s\in\R$,  the asymptotic variance 
 \[ \sigma_{\phi}^2 = \lim_{T\rightarrow\infty}\frac{1}{T}\mathbf{E}\left[\int_0^T\left(\phi(\Phi_{s, s+ t}(w))- \big\langle\mu_{s+t},\phi \big\rangle \right)dt\right]^2,\]
exits and  is independent of $s$ and $w$.

\vskip0.05in

(2). For any integer $p\geq 2$, $\eta\in(0, 2^{-p-5}\eta_0]$, and $\phi\in C^{\g}_{\eta}(H)$ with $\s_{\phi}^2>0$, and $w \in H$,  there are  constants $C_{p} = C_{p}(\|\phi\|_{\g,\eta}, \|w\|)>0$ and $T_0>0$ such that for all $T\geq T_0$, 
\begin{align*}
\sup _{z \in \mathbb{R}}\left|\mathbf{P}\left\{\frac{1}{\sqrt{T}}\int_0^T\Big(\phi(\Phi_{s, s+t}(w))- \langle\mu_{s+t},\phi \rangle\Big)dt \leq z\right\}-\N_{\sigma_{\phi}}(z)\right| \leq C_{p}T^{-\frac{2^{p-1}\overline{\g}_0}{(2^p+1)(A+n)}},
\end{align*}

\vskip0.05in

(3). For $\eta\in(0, 2^{-7}\eta_0]$ and $\phi\in C^{\g}_{\eta}(H)$ such that $\s_{\phi}^2 = 0$, and $w\in H$, there is a constant $C = C(\|\phi\|_{\g,\eta}, \|w\|)>0$ such that for all $T\geq 1$, 
\begin{align*}
\sup _{z \in \mathbb{R}}\left(|z|\wedge 1\right)\left|\mathbf{P}\left\{\frac{1}{\sqrt{T}}\int_0^T\Big(\phi(\Phi_{s, s+t}(w))- \langle\mu_{s+t},\phi \rangle\Big)dt \leq z\right\}-\N_{0}(z)\right| \leq CT^{-\frac{\overline{\g}_0}{2(A+n)}}.
\end{align*}
\end{theorem}

\vskip0.05in

%
%

\section{Exponential mixing}\label{section-ergodicmixing}

In this section, we will prove Theorem \ref{ergodicmixing}.  We first give the following Theorem \ref{contractiontransition} that ensures the contraction of the family $\mathcal{P}_{s, s+t, h}^{*}$.  

\vskip0.05in

\begin{theorem}\label{contractiontransition}
 Assume the standing assumption \eqref{assumption}. There exists $\eta_0>0$ such that for $\eta\in(0,\eta_0]$, there are positive constants $C$ and $\varpi$ such that
\begin{align}\label{eqcontraction}
\rho(\mathcal{P}_{s, s+t, h}^{*}\mu_1,\mathcal{P}_{s, s+t, h}^{*}\mu_2)\leq Ce^{-\varpi t}\rho(\mu_1,\mu_2),
\end{align}
for every $s\in \mathbb{R}$, $t\geq 0$, $h\in \To^n$ and any $\mu_1, \mu_2\in \mathcal{P}(H)$.
\end{theorem}

In view of the translation identity \eqref{translation-identity}, we only need to prove the theorem for $s=0$. This is proved by adapting the Harris-like theorem in \cite{HM08}, consisting of three ingredients: the Lyapunov structure, the gradient estimate and the uniform irreducibility over bounded sets. Theorem \ref{contractiontransition} immediately follows from the scheme presented in subsection 3.1 of \cite{HM08} and the cocycle property of  $\mathcal{P}_{s, s+t, h}^{*}$ from \eqref{cocycle-property-dual},  once one obtains the three ingredients. In particular, the Lyapunov structure and the gradient estimate follow from \cite{HM08,HM11}. The uniform irreducibility over bounded sets follows from the approximate controllability \cite{AS05, AS06,GHM18}  and standard compactness arguments. We omit the details for the proofs in the current time inhomogeneous setting, since there is no significant difference because the bounds on solutions do not depend on the time symbols $h$. Interested readers are referred to \cite{Liu22} for a verification.

\vskip0.05in

We now prove Theorem \ref{ergodicmixing} by applying a fixed point argument and the contraction \eqref{eqcontraction}, based on the time symbols formulation from subsection \ref{skew-product-homogenization}. We first introduce a semigroup $S_t$ acting on $C(\To^n, \P_{1}(H))$ by lifting the cocycle $\P_{0, t, h}^{*}$ through a pullback procedure. To accommodate the non-uniformity of the dynamics, we apply the fixed point theorem for $S_t$  acting on a family of nested closed subsets of $C(\To^n, \P_{1}(H))$, which possess a common fixed point $\G$. The contraction of $S_t$ is guaranteed by \eqref{eqcontraction}. The composition of $\G$ with the particular irrational rotation orbit $\b_t0$ gives the quasi-periodic invariant measure $\mu_t:=\G_{\b_t0}$ of the original system $\eqref{NS}$.

\vskip0.05in

\vskip0.05in

Note that Theorem \ref{ergodicmixing} follows from the following Theorem \ref{fixedpoint} with $\mu_s = \Gamma_{\beta_s0}$ by taking $h = \beta_s0$ in \eqref{mixing-PH},  and applying the translation identity \eqref{translation-identity}. 

\vskip0.05in

 \begin{theorem}\label{fixedpoint}
 There is a unique map $\Gamma\in C(\To^n, \mathcal{P}(H))$, such that $\mathcal{P}_{0, t, h}^{*}\Gamma_{h} = \Gamma_{\b_{t}h}$ for any $h\in\To^n, t\geq 0$. Moreover, there is a constant $\eta_0>0$, such that for every $\eta\in (0, \eta_0/2]$, there are constants $C, \varpi>0$, such that $\Gamma\in C(\To^n, \mathcal{P}_1(H))$ and
\begin{align}\label{mixing-PH}
\rho(\mathcal{P}_{0, t, h}^{*}\mu, \Gamma_{\b_{t}h})\leq Ce^{-\varpi t}\rho(\mu, \Gamma_{h}),\quad  t\geq 0, \mu\in\mathcal{P}(H), h\in\To^n,
\end{align}
where $C, \varpi$ do not depend on $ h$.
Also $\int_H \exp\left(2\eta\|w\|^2\right)\Gamma_{h}(dw)\leq C$ for all $h\in\To^n$.\\
Furthermore, $\G\in C^{\z}(\To^n, (\P_1(H),\r))$ if $f\in C^{\g}(\R, H_2)$, where $\z = \frac{\varpi \g}{r_0+\varpi}$ with $r_0$ from Lemma \ref{bounds} in the Appendix. 
\end{theorem}

\begin{proof}

Recall that $\mathcal{P}_1(H)$ is defined by \eqref{P1H}. The proof is divided into the following five steps. 

\vskip0.05in

{\it Step 1: Lifting the cocycle $\mathcal{P}_{0, t, h}^{*}$ to a semigroup.} By estimate \eqref{eq: est1}, for any $t\geq 0$, $\mathcal{P}_{0, t, h}^*$ maps $\mathcal{P}_1(H)$  to itself. Denote for convenience
\[\varphi:  \mathbb{R}_{+}\times\mathcal{P}_1(H)\times \To^n\rightarrow \mathcal{P}_1(H), \text{ by }\varphi(t, \mu, h) = \mathcal{P}_{0, t, h}^{*}\mu.\]
 It follows from the translation identity \eqref{translation-identity} that $\varphi$ has the cocycle property over the base dynamical system $(\To^n, \mathbb{R}, \beta)$ since for all $\tau, t\geq 0 $ and $h\in\To^n$, $\mu\in \mathcal{P}_1(H)$,
\[\mathcal{P}_{0, t+\tau, h}^{*}\mu= \mathcal{P}_{t, t+\tau, h}^{*}\mathcal{P}_{0, t, h}^{*}\mu = \mathcal{P}_{0, \tau,\b_th}^{*}\mathcal{P}_{0, t, h}^{*}\mu.\]
Hence the pull-back map $S^t$ induced from $\varphi$, which is defined on the space of quasi-periodic graphs $C(\To^n, \mathcal{P}_1(H))$ as 
\[S^t(\gamma)(h) : = \varphi(t, \gamma(\b_{-t}h), \b_{-t}h)), \, \gamma\in C(\To^n, \mathcal{P}_1(H)), \]
satisfies the semigroup property $S^{t_1}S^{t_2}\gamma(h) = S^{t_1+t_2}\gamma(h)$ by a straightforward verification. 

\vskip0.05in

{\it Step 2: Choosing appropriate spaces to apply the fixed point theorem.} We would like to apply the fixed point theorem for $S^t$ on $C(\To^n, \mathcal{P}_1(H))$ endowed with the metric (which is complete since $(\mathcal{P}_1(H),\rho)$ is complete.) 
\[p(\gamma_1, \gamma_2) := \max_{h\in \To^n}\rho(\gamma_1(h), \gamma_2(h)),\,\gamma_1,\gamma_2 \in C(\To^n, \mathcal{P}_1(H)).\]
However, the continuity of  $\varphi(t, \mu, h)$ with respect to $(\mu, h)$ is not straightforward due to the Lyapunov structure of the solution of \eqref{NS}. Hence $C(\To^n, \mathcal{P}_1(H))$ may not be invariant under the map $S^t$. 

\vskip0.05in

Indeed, from the definition of $\rho$ as in \eqref{rho-introduction}, one has
\begin{align}\label{continuous02}
\r(w_1, w_2)\leq \|w_1-w_2\|\left(e^{\eta\|w_1\|^2}+e^{\eta\|w_2\|^2} \right), \quad \forall w_1,w_2\in H.
\end{align}
It is known \cite{Ch04,Vil08} that for any $\mu_1, \mu_2\in \mathcal{P}(H)$,
\begin{align}\label{continuous03}
\rho(\mu_1, \mu_2) = \inf\mathbf{E}\rho(X_1, X_2),
\end{align}
where the infimum is taken over all couplings $(X_1, X_2)$ for $(\mu_1, \mu_2)$.
Combining \eqref{continuous02}-\eqref{continuous03}, H\"older's inequality, estimates \eqref{eq: est1} and \eqref{continuousonhull}, it follows that
\begin{align*}
&\rho(\mathcal{P}_{0, t, h_1}^*\delta_{w},\mathcal{P}_{0, t, h_2}^*\delta_{w} )\leq \mathbf{E}\rho(w_{0, t, h_1}(w), w_{0, t, h_2}(w))\\
&\leq \left(\mathbf{E}\|w_{0, t, h_1}(w)-w_{0, t, h_2}(w)\|^2\right)^{\frac12}\left(2\mathbf{E}\left[\exp({2\eta \|w_{0, t, h_1}(w)\|^2})+\exp({2\eta \|w_{0, t, h_2}(w)\|^2})\right]\right)^{\frac12}\\
&\leq C e^{r_0 t}g(w)\sup_{t\in\R}\|\Psi(\beta_th_1)-\Psi(\beta_th_2)\|, 
\end{align*}
where $r_0$ is from \eqref{continuousonhull}, and $g(w) = \exp\left({2\eta}\|w\|^2\right)$.
Therefore we obtain 
\begin{align}\label{continuousiny}
\rho(\mathcal{P}_{0, t, h_1}^*\mu,\mathcal{P}_{0, t, h_2}^*\mu)\leq Ce^{r_0 t}\int_Hg(w)\mu(dw)\sup_{t\in\R}\|\Psi(\beta_th_1)-\Psi(\beta_th_2)\|.
\end{align}
It is unclear if each $\mu\in\mathcal{P}_1(H)$ yields a finite $\int_Hg(w)\mu(dw)$, therefore we confine ourselves to those measures that make the integral finite to ensure the continuity. 

\vskip0.05in

To be specific, consider the family of closed subsets of $\mathcal{P}_1(H)$,
\[\mathcal{P}_R: =\{\mu\in\mathcal{P}(H): \int_{H}g(w)\mu(dw)\leq R\}, \quad R>0.\]
By the contraction property in Theorem \ref{contractiontransition},  we have that for any $R>0$ and $\mu\in \mathcal{P}_{R}$,  $\varphi$ is continuous in $\mu$, uniformly with respect to $h$. And by inequality \eqref{continuousiny},  it is continuous in $h$ uniformly for $\mu$. Hence $\varphi$ is joint continuous in $(\mu, h)\in\mathcal{P}_{R}\times\To^n$. 
Then the fixed point argument will be applied on the complete subset $C(\To^n, \P_R)$. However  the trade off for the continuity is the loss of the invariance of $C(\To^n, \P_R)$ under $S^t$ uniformly for any $t\geq 0$. Indeed, it follows from \eqref{eq: est1} that  for $\mu\in \mathcal{P}_R$, and any $t\geq 0$,
\begin{align*}
\int_H\mathcal{P}_{0, t, h} g(w)\mu(dw) &= \int_H \mathbf{E} g(\mathrm{\Phi}_{0, t, h}(w))\mu(dw)\\
&\leq C \int_H g^{\alpha(t)}(w)\mu(dw)\leq C \left( \int_H g(w)\mu(dw)\right)^{\alpha(t)}\leq CR^{\alpha(t)},
\end{align*}
with $\a(t) = e^{-\nu t}$, where we used Jensen's inequality in the penultimate step. One can check that it is impossible to choose a common $R>0$ such that  $CR^{\alpha(t)}\leq R$ for any $t\geq 0$ since $\alpha(t)\ra 1$ as $t\ra 0$. However, note that for each fixed $t_0>0$, if we choose $R = R_{t_0}: =C^{\frac{1}{1-\alpha(t_0)}}$ then $CR^{\alpha(t)}\leq CR^{\alpha(t_0)}= R_{t_0}$, which gives the invariance under $S^t$ uniformly for $t\geq t_0$. 

\vskip0.05in

{\it Step 3: Contraction and construction of the invariant measure as a fixed point.} Now for any fixed $t_0\in(0,1)$,  the above analysis shows that the map $S^t : C(\To^n, \mathcal{P}_{R_{t_0}}) \rightarrow C(\To^n, \mathcal{P}_{R_{t_0}}) $ is well defined for $t\geq t_0$. It remains to show that it is a contraction. Indeed, by Theorem \ref{contractiontransition}, one has
\begin{align*}
p(S^{t}\gamma_1,S^{t}\gamma_2) &= \max_{h\in {\To^n}}\rho(\varphi(t, \gamma_1(\b_{-t}h),\b_{-t}h),\varphi(t, \gamma_2(\b_{-t}h),\b_{-t}h))\\
&= \max_{h\in {\To^n}}\rho(\mathcal{P}_{0,t, \b_{-t}h}^{*}\gamma_1(\b_{-t}h), \mathcal{P}_{0,t, \b_{-t}h}^{*}\gamma_2(\b_{-t}h))\\
&\leq Ce^{-\varpi t}\max_{h\in {\To^n}}\rho(\gamma_1(\b_{-t}h),\gamma_2(\b_{-t}h)) = Ce^{-\varpi t}p(\gamma_1,\gamma_2).
\end{align*}
Therefore for large $T>t_0$, $p(S^{T}\gamma_1,S^{T}\gamma_2)\leq c  p(\gamma_1,\gamma_2)$ for some $c\in(0,1)$. Fix such a $T$, then $S^T$ is a contraction over the complete metric space $C({\To^n}, \mathcal{P}_{R_{t_0}})$, so there is a unique fixed point  $\Gamma_{t_0}\in C({\To^n}, \mathcal{P}_{R_{t_0}})$ of $S^T$.  Noting for any $t\geq t_0$, $S^t$ maps $C({\To^n}, \mathcal{P}_{R_{t_0}})$ to itself, hence
\[S^T(S^t\Gamma_{t_0}) = S^t(S^T\Gamma_{t_0}) = S^t(\Gamma_{t_0})\]
implies that $S^t(\Gamma_{t_0}) = \Gamma_{t_0}$ by the uniqueness of the fixed point, which shows that $\Gamma_{t_0}$ is a fixed point of $S^t$ for $t\geq t_0$.  For $0<t_1\leq t_0$, one has $R_{t_1}\geq R_{t_0}$, so $\mathcal{P}_{R_{t_0}}\subset \mathcal{P}_{R_{t_1}}$. And for the same $T>0$, $S^T$ is a contraction on $C({\To^n}, \mathcal{P}_{R_{t_1}})$, which has a unique fixed point $\Gamma_{t_1}$. By the uniqueness, $\Gamma_{t_1} = \Gamma_{t_0}$, hence $\Gamma_{t_0}$ is also a fixed point of $S^t$ for $t\geq t_1$. Since $t_1$ is arbitrary, we see that $\Gamma: = \Gamma_{t_0}$ is a fixed point of $S^t$ for $t\geq 0$, that is, $\varphi(t, \Gamma(\b_{-t}h),\b_{-t}h) = \Gamma(h)$ for all $h\in \To^n$. Replacing $h$ with $\b_th$ we have $\varphi(t, \Gamma(h),h) = \Gamma(\b_tu)$ which by definition is
\[\mathcal{P}_{0, t, h}^{*}\Gamma(h)=\Gamma(\b_th).\]
Hence the invariance follows, and the exponential mixing \eqref{mixing-PH} then follows from the invariance and Theorem \ref{contractiontransition}. Note that by replacing $h$ with $\b_sh$ in the invariance identity and using the translation identity \eqref{translation-identity}, we have 
\begin{align}
\mathcal{P}_{s, s+t, h}^{*}\Gamma(\b_sh)=\Gamma(\b_{s+t}h), \quad\forall s\in\R, t\geq 0, h\in \To^n. 
\end{align}

\vskip0.05in

{\it Step 4: Uniqueness of invariant measure on $C(\To^n, \P(H))$.} The uniqueness is essentially due to exponential stability of the fixed point. However, we need to first show that any invariant measure in $C(\To^n, \P(H))$ actually lives in $C(\To^n, \P_1(H))$  to ensure the finiteness of Wasserstein metric.  Suppose that there is another $\widetilde{\Gamma}\in C(\To^n, \P(H))$ that is invariant. 

We claim that $\widetilde{\Gamma}\in C(\To^n, \P_1(H))$. Indeed, 
for $R>0$, let 
\begin{align*}
g_{R}(w) = \left\{
\begin{array}{rr}
e^{2\eta\|w\|^2}, \text{ if } \|w\|\leq R,\\
e^{2\eta R^2}, \text{ if } \|w\|\geq R.\\
\end{array}
\right.
\end{align*}
Then by the invariance of $\widetilde{\Gamma}$ and estimate \eqref{eq: est1}, we have for any $M, N>0$, 
\begin{align*}
\int_{H}g_{R}(w)\widetilde{\Gamma}_{h}(dw)& = \int_{H}\mathcal{P}_{-N, 0, h}g_{R}(w)\widetilde{\Gamma}_{\b_{-N}h}(dw)\\
&\leq \int_{\{\|w\|\leq M\}}\mathcal{P}_{-N, 0, h}g_{R}(w)\widetilde{\Gamma}_{\b_{-N}h}(dw)+ \int_{\{\|w\|\geq M\}}\mathcal{P}_{-N, 0, h}g_{R}(w)\widetilde{\Gamma}_{\b_{-N}h}(dw)\\
&\leq \int_{\{\|w\|\leq M\}}\mathbf{E}g(w_{-N, 0, h}(w))\widetilde{\Gamma}_{\b_{-N}h}(dw)+ e^{2\eta R^2}\widetilde{\Gamma}_{\b_{-N}h}(\{\|w\|\geq M\})\\
&\leq Ce^{2\eta\alpha(N) M^2}+ e^{2\eta R^2}\widetilde{\Gamma}_{\b_{-N}h}(\{\|w\|\geq M\}).
\end{align*}
Since $\To^n$ is compact, and $\widetilde{\Gamma}\in C(\To^n, \P(H))$, where $\P(H)$ is endowed with the topology of weak convergence, therefore $\{\widetilde{\Gamma}_{h}\}_{h\in\To^n}$ is compact and hence tight by Prokhorov's theorem: 
 for any $\varepsilon>0$, there is a compact subset $K_{\varepsilon}$ of $H$ such that 
\[\widetilde{\G}_{h}(H\backslash K_{\varepsilon})<\varepsilon, \quad \forall h\in\To^n.\]
Hence for any $R>0$ and $\varepsilon = e^{-2\eta R^2}$, there is a compact subset $K_{\varepsilon}$ of $H$ such that 
\[\widetilde{\Gamma}_{\b_{-N}h}(H\backslash K_{\varepsilon})<\varepsilon, \quad\forall N>0.\]
Now we can choose $M$ large enough such that $K_{\varepsilon}\subset \{\|w\|\leq M\}$ so that 
\[e^{2\eta R^2}\widetilde{\Gamma}_{\b_{-N}h}(\{\|w\|\geq M\})\leq 1.\]
Since $\alpha(N) \ra 0$ as $N\ra\infty$, we can choose $N$ large such that $e^{2\eta\alpha(N) M^2}\leq 1$ as well. Therefore we have
\[\int_{H}g_{R}(w)\widetilde{\Gamma}_{h}(dw)\leq C, \quad\forall R>0,\]
which, by the monotone convergence theorem, in turn implies that 
\[\int_{H}g(w)\widetilde{\Gamma}_{h}(dw)\leq C, \quad\forall h\in\To^n,\]
and hence $\widetilde{\Gamma}\in C(\To^n, \P_1(H))$ and the  claim is proved.  This ensures that 
\[\sup_{h\in\To^n}\r(\Gamma(h),\widetilde{\Gamma}(h))<\infty.\]
Now by the translation identity \eqref{translation-identity} and Theorem \ref{contractiontransition}, we have for $h\in\To^n, t\geq s$, 
\begin{align*}
\r(\Gamma(\b_{t}h), \widetilde{\Gamma}(\b_{t}h)) &=\r(\P_{s, t, h}^*\Gamma(\b_sh), \P_{s, t, h}^*\widetilde{\Gamma}(\b_sh))\\
&\leq Ce^{-\varpi(t-s)}\r(\Gamma(\b_sh),\widetilde{\Gamma}(\b_sh) )\leq C\sup_{h\in\To^n}\r(\Gamma(h),\widetilde{\Gamma}(h))e^{-\varpi (t-s)}. 
\end{align*}
By letting $s\ra-\infty$, it follows that $\Gamma(\b_{t}h) = \widetilde{\Gamma}(\b_{t}h)$ for $t\in \mathbb{R}$. It particular this is true for $t=0$ and any $h\in\To^n$, hence $\Gamma =  \widetilde{\Gamma}$. 

\vskip0.05in

{\it Step 5: H\"older regularity of the fixed point.} Note that $f\in C^{\g}(\R, H_2)$ implies $\Psi\in C^{\g}(\To^n, H)$.  The regularity is due to exponential contraction in Theorem \ref{contractiontransition} and the H\"older regularity of the quasi-periodic force. Indeed,  to show that $\G\in C^{\z}(\To^n, (\P_1(H),\r))$ if $\Psi\in C^{\g}(\To^n, H)$, where $\z = \frac{\varpi \g}{r_0+\varpi}$ with $r_0$ from the Appendix, observing that for any $t\geq0$ and $h_1, h_2\in \To^n$, by the invariance of $\G$ and estimate \eqref{continuousonhull},  one has 
\begin{align*}
 &\r\Big(\Gamma(h_1) , \Gamma(h_2)\Big) =\r\Big(\varphi\left(t, \Gamma(\b_{-t}h_1),\b_{-t}h_1\right), \varphi\left(t, \Gamma(\b_{-t}h_2),\b_{-t}h_2\right)\Big)\\
 &\leq \r\Big(\varphi\left(t, \Gamma(\b_{-t}h_1),\b_{-t}h_1\right), \varphi\left(t, \Gamma(\b_{-t}h_2),\b_{-t}h_1\right)\Big)+\r\Big(\varphi\left(t, \Gamma(\b_{-t}h_2),\b_{-t}h_1\right), \varphi\left(t, \Gamma(\b_{-t}h_2),\b_{-t}h_2\right)\Big)\\
 &\leq Ce^{r_0 t}\int_Hg(w)\G(\b_{-t}h_1)(dw)\|\Psi\|_{\g}|h_1 -h_2|^{\g}+ Ce^{-\varpi t} \r\Big(\G(\b_{-t}h_1),\G(\b_{-t}h_2) \Big)\\
 &\leq Ce^{r_0t}R_{t_0}\|\Psi\|_{\g}|h_1 -h_2|^{\g} + Ce^{-\varpi t}\sup_{h_1, h_2\in \To^n} \r\Big(\G(h_1),\G(h_2)\Big)\\
 &\leq C(e^{r_0t}|h_1 -h_2|^{\g} + e^{-\varpi t})\leq C|h_1-h_2|^{\zeta}, 
\end{align*}
with $\z = \frac{\varpi \g}{r_0+\varpi}$, by applying the following lemma. 

\begin{lemma}\label{choosing-Holder-exponent}
For $D\geq 1, \L_1, \L_2>0, \g\in(0, 1], 0<\d\leq D$, one has
\begin{align*}
e^{\L_1T}\d^{\g}+e^{-\L_2 T}\leq 2D^{\g} \d^{\overline{\g}},
\end{align*}
for $\overline{\g} = \frac{\L_2}{\L_1+\L_2}\g$, by choosing $T = -\frac{\g}{\L_1+\L_2}\ln\d$ for $\d<1$ and $T=0$ for $\d\geq 1$.
\end{lemma}
The proof is then complete. 
\end{proof}

It turns out that Theorem \ref{fixedpoint} also implies the convergence of time averages of the transition probabilities, which is useful when applied to the proof of the limit theorems in the next subsection. 
\begin{proposition}\label{quasi-periodic-convergence-of-discrete-average}
For any $(w_0,h)\in H\times\mathbb{T}^n$ and $K\in \mathbb{N}$, we have the following weak convergence of measures: 

\vskip0.05in

1. $\ds\frac1N\sum_{j=1}^N\mathcal{P}^*_{0, (j-1)K, h}\delta_{w_0}\ra \int_{\To^n}\G_{g}\lambda(dg).$
\vskip0.05in 
2. $\ds\frac1N\sum_{j=1}^NP^*_{(j-1)K}\delta_{(w_0, h)}\ra\G_g(dw)\lambda(dg)$ and $\ds\frac1T\int_0^TP^*_t\delta_{(w_0, h)}dt\ra \G_g(dw)\lambda(dg)$ as well. \\
Here $\l$ is the invariant Lebesgue measure of the irrational rotation. 
\end{proposition}

\begin{proof}
For any $\phi\in\mathrm{Lip}_{\r}(H)$, we have by the Monge-Kantorovich duality \eqref{Monge-Kantorovich duality}, the invariance and mixing of the quasi-periodic invariant measure from Theorem \ref{fixedpoint} that 
\begin{align*}
&\left|\left\langle\frac1N\sum_{j=1}^N\left(\mathcal{P}^*_{0, (j-1)K, h}\delta_{w_0}-\G_{\b_{(j-1)K}h}\right), \phi\right\rangle\right|\\
&\leq \mathrm{Lip}_{\r}(\phi)\frac1N\sum_{j=1}^N\r({P}^*_{0, (j-1)K, h}\delta_{w_0}, {P}^*_{0, (j-1)K, h}\G_{h})
\leq C\mathrm{Lip}_{\r}(\phi)\frac1N \sum_{j=1}^Ne^{-\varpi(j-1)K}\r(\delta_0, \G_{h}),
\end{align*}
which tends to $0$ as $N\ra\infty$. Also by Birkhoff's ergodic theorem for the irrational rotation on $\mathbb{T}^n$, one has
\[\left\langle\frac1N\sum_{j=1}^N\G_{\b_{(j-1)K}h}, \phi\right\rangle\ra \int_{\mathbb{T}^n}\langle\G_{g}, \phi\rangle\lambda(dg).\]
Hence the first claim of the lemma follows. 

\vskip0.05in

Now let $\phi\in\mathrm{Lip}_{\r,d}(H\times \mathbb{T}^n)$, where $H$ is equipped with the metric $\r$ and $d$ is the usual distance in $\mathbb{T}^n$ induced from $\mathbb{R}^n$. Observe that 
\begin{align*}
&\left\langle\frac1N\sum_{j=1}^NP^*_{(j-1)K}\delta_{(w_0, h)}, \phi\right\rangle  = \frac1N\sum_{j=1}^N\mathcal{P}_{0,(j-1)K, h}\phi(\cdot, \b_{(j-1)K}h)(w_0)\\
&=\frac{1}{N}\sum_{j=1}^N\Big\langle \mathcal{P}_{0, (j-1)K, h}^*\delta_{w_0} - \mathcal{P}_{0, (j-1)K, h}^*\G_h, \phi(\cdot, \b_{(j-1)K}h)\Big\rangle + \frac{1}{N}\sum_{j=1}^N\Big\langle \mathcal{P}_{0, (j-1)K, h}^*\G_h, \phi(\cdot, \b_{(j-1)K}h)\Big\rangle\\
&: = I + II,
\end{align*}
where the first term in the sum vanishes by the mixing of the quasi-periodic invariant measure since
\begin{align*}
|I| \leq \frac{1}{N}\sum_{j=1}^N\mathrm{Lip}_{\r}\phi(\cdot, \s_{(j-1)K}h)\r(\mathcal{P}_{0, (j-1)K, h}^*\delta_w,  \mathcal{P}_{0, (j-1)K, h}^*\G_h)\leq  \frac{C\mathrm{Lip}_{\r, d}(\phi)}{N}\sum_{j=1}^Ne^{-\varpi(j-1)K}\r(\delta_{w_0}, \G_h)\ra0.
\end{align*}
The second term converges to the average of $\phi$ with respect to $\G_h(dw)\lambda(dg)$ by Birkhoff's ergodic theorem for the irrational rotation $\b$:
\begin{align*}
II = \frac{1}{N}\sum_{j=1}^N\Big\langle \G_{\s_{(j-1)K}h}, \phi(\cdot, \b_{(j-1)K}h)\Big\rangle\ra \int_{H\times\mathbb{T}^n}\phi(w, g)\G_g(dw)\lambda(dg),
\end{align*}
since the observable function $g\in\mathbb{T}^n\ra\Big\langle \G_{g}, \phi(\cdot, g)\Big\rangle$ is continuous. The proof for the  continuous time version is similar. 
\end{proof}

\section{Limit theorems with convergence rates}\label{sectionlimittheorem}
Let $\eta_0, r_0$ be the constants from Lemma \ref{bounds} in the Appendix. 
In this section, we prove Theorem \ref{rate-of-convergence-CLT} on the quantitative limit theorems  with explicit convergence rates for the time inhomogeneous  solution processes of the Navier-Stokes equation \eqref{NS}. In fact, we will prove these results for more general class of observable functions that will be given below. The proof is  based on a particular martingale approximation scheme designed for the inhomogeneous  solution processes. Subsection \ref{subsection-martingale-approximation} below is devoted to the study of the martingale approximation scheme. In subsection \ref{subsection-asymptotic-variance} we investigate the asymptotic variance and its properties. The quantitative limit theorems will be proved in the last subsection \ref{subsection-rate-of-convergence}. 

\vskip0.05in

We now define the space of observable functions and state the main results of this section that include Theorem \ref{rate-of-convergence-CLT} as a particular case. For $\g\in (0, 1]$, let  $C^{\g}_{\eta, H}(H\times\To^n)$ be the space of Hölder continuous functions weighted by the Lyapunov function $e^{\eta\|w\|^2}$, 
\begin{align}\label{weighted-holder-space}
C^{\g}_{\eta, H}(H\times\To^n): =\left\{ \phi\in C(H\times\To^n) : \|\phi\|_{\g,\eta, H}<\infty\right\},
\end{align}
where 
\begin{align}\label{weighted-holder-norm}
\|\phi\|_{\g,\eta, H}:=\sup_{(w,h)\in H\times\To^n}\frac{|\phi(w, h)|}{e^{\eta\|w\|^2}} + \sup_{\substack{h\in\To^n\\0<\|w_1-w_2\|\leq1}}\frac{|\phi(w_1, h)-\phi(w_2, h)|}{\|w_1-w_2\|^{\g}\left(e^{\eta\|w_1\|^2}+e^{\eta\|w_2\|^2}\right)}.
\end{align}
Let also 
$C_{\eta, \To^n}^{\g}(H\times\To^n)$ be the space of functions that are Hölder continuous on $\To^n-$component uniformly  on bounded subset of $H$
\begin{align}\label{weighted-holder-space-uniformly-Hölder-on-bounded-subsets}
C_{\eta,\To^n}^{\g}(H\times\To^n): =\left\{ \phi\in C(H\times\To^n) : \|\phi\|_{\g, \eta, \To^n} <\infty\right\},
\end{align}
where 
\[\|\phi\|_{\g, \eta, \To^n}: = \sup_{(w,h)\in H\times\To^n}\frac{|\phi(w, h)|}{e^{\eta\|w\|^2}} + \sup_{\substack{w\in H\\0<\|h_1-h_2\|\leq1}}\frac{|\phi(w, h_1)-\phi(w, h_2)|}{e^{\eta\|w\|^2}|h_1-h_2|^{\g}}.\]

For any $\phi\in C^{\g}_{\eta, H}(H\times\To^n)$, we set $\widetilde{\phi}$ as the associated function obtained by normalizing $\phi$ with the quasi-periodic invariant measure, 
\begin{align}\label{widetilde-phi}
\widetilde{\phi}(w, h) = \phi(w, h) - \langle\G_h, \phi(\cdot, h)\rangle.
\end{align}
It is clear that $\widetilde{\phi}\in  C^{\g}_{\eta, H}(H\times\To^n)$ as well.  Recall the homogenized process $X_t(w, h) = (\Phi_{0, t, h}(w), \b_t h)$ as in \eqref{homogenized-process}. We have 
\begin{align*}
\widetilde{\phi}(X_t(w, h)) = \phi(\Phi_{0, t, h}(w), \b_t h) - \langle\G_{\b_th}, \phi(\cdot, \b_th)\rangle.
\end{align*}

\vskip0.05in

Recall from Theorem \ref{fixedpoint} that when $h = 0$, $\G_{\b_t0} = \mu_t$ is the unique quasi-periodic invariant measure of the Navier-Stokes system \eqref{NS} with the deterministic force $f(t, x)$. And when $\phi$ is an observable function on $H$,  we have $\widetilde{\phi}(X_t(w, 0)) = \phi(\Phi_{0, t, h}(w)) -  \langle\mu_t, \phi\rangle$, which is the observation along the solution process normalized by the quasi-periodic invariant measure. In particular, Theorem \ref{rate-of-convergence-CLT} is obtained from the following Theorems \ref{rate-of-convergence-SLLN-proof}-\ref{rate-of-convergence-CLT-proof} by taking the observable function $\phi\in C_{\eta}^{\g}(H)\subset C_{\eta, H}^{\g}(H\times\To^n)$ and $h = 0$. For simplicity, Theorems \ref{rate-of-convergence-SLLN-proof}-\ref{rate-of-convergence-CLT-proof} are proved for initial time $s = 0$, while the proof applies to $s\neq 0$ without any change. 

\vskip0.05in

The first result of this section is the following strong law of large numbers with a convergence rate. 
\begin{theorem}\label{rate-of-convergence-SLLN-proof}
Let $\varepsilon>0$, for every integer $p\geq 3$ satisfying $2^p>1/\varepsilon$, every $\eta\in (0, 2^{-p-1}\eta_0]$, and every $\phi\in C^{\g}_{\eta, H}(H\times\To^n)$, $(w, h)\in H\times\To^n$, there is an almost surely finite random time $T_0(\o) \geq 1$, depending on $p, \varepsilon, \|\phi\|_{\g,\eta, H}, \|w\|, h$ such that for all $T>T_0$,  we have 
\begin{align*}
\left|\frac{1}{T}\int_0^T\Big(\phi(X_{t}(w, h))- \big\langle \G_{\b_{t}h}, \phi(\cdot, \b_{t}h)\big\rangle\Big)dt\right|\leq T^{-\frac12+\varepsilon}, 
\end{align*}
Moreover, for every $0<\ell<\min\{2^p\varepsilon-1, 2^{p-2}-1\}$, there is a constant $C_p = C_p(\|\phi\|_{\g,\eta, H}, \ell, \varepsilon)$ such that 
\[\mathbf{E}T_0^{\ell}\leq C_pe^{2^{p+1}\eta\|w\|^2}. \]
\end{theorem}

\vskip0.05in

The second result is the central limit theorem with a convergence rate. Let $\N_{\sigma}$ be the distribution function of the normal random variable $N(0, \s^2)$. 
\begin{theorem}\label{rate-of-convergence-CLT-proof}
Assume $\Psi \in C^{\g}(\To^n, H)$ and the frequency $\alpha$ satisfies the Diophantine condition \eqref{Diophantine condition} with constant $A$ and dimension $n$. Let $\overline{\g}_0 =\left( \frac{\varpi \g}{5(r_0+\varpi)}\right)^3$, where $\varpi$ is the mixing rate from Theorem \ref{ergodicmixing} and $r_0$ is the constant from Lemma \ref{bounds} in the Appendix.

\vskip0.05in

1. For any $\eta\in(0,\eta_0/16]$, $(w, h) \in H\times\To^n$ and $\phi\in C^{\g}_{\eta, H}(H\times\To^n)$,  the asymptotic variance 
 \[ \sigma_{\phi}^2 =\lim_{T\ra\infty}\frac{1}{T}\mathbf{E}\left(\int_0^T\widetilde{\phi}\left(X_t(w,h)\right)dt\right)^2,\]
exits and  is independent of $(w, h)$.

\vskip0.05in

2. For any integer $p\geq 2$, $\eta\in(0, 2^{-p-5}\eta_0]$, and $\phi\in C^{\g}_{\eta, H}(H\times\To^n)$ with $\s_{\phi}^2>0$, and $(w, h) \in H\times\To^n$, there are  constants $C_{p} = C_{p}(\|\phi\|_{\g,\eta, H},\|\phi\|_{\g,\eta, \To^n}, \|w\|)>0$ and $T_0>0$ such that for all $T\geq T_0$, 
\begin{align*}
\sup _{z \in \mathbb{R}}\left|\mathbf{P}\left\{\frac{1}{\sqrt{T}}\int_0^T\widetilde{\phi}(X_t(w, h))dt \leq z\right\}-\N_{\sigma_{\phi}}(z)\right| \leq C_{p} \left(T^{-\frac14}+T^{-\frac{2^{p-2}}{2^p+1}} +T^{-\frac{2^{p-1}\overline{\g}_0}{(2^p+1)(A+n)}}\right),
\end{align*}

\vskip0.05in

3. For $\eta\in(0, 2^{-7}\eta_0]$ and $\phi\in C^{\g}_{\eta, H}(H\times\To^n)$ such that $\s_{\phi}^2 = 0$, and $(w, h)\in H\times\To^n$, there is a constant $C = C(\|\phi\|_{\g,\eta, H}, \|\phi\|_{\g,\eta, \To^n}, \|w\|)>0$ such that for all $T\geq 1$, 
\begin{align*}
\sup _{z \in \mathbb{R}}\left(|z|\wedge 1\right)\left|\mathbf{P}\left\{\frac{1}{\sqrt{T}}\int_0^T\widetilde{\phi}(X_t(w_0, h_0))dt \leq z\right\}-\N_{0}(z)\right| \leq C\left(T^{-\frac14}+T^{-\frac{\overline{\g}_0}{2(A+n)}}\right)
\end{align*}

\end{theorem}

\subsection{The martingale approximation}\label{subsection-martingale-approximation}
 
We first give several properties of the spaces of observable functions defined above in the following Proposition \ref{properties-of-function-spaces}.  Then we prove a mixing result (as a consequence of Theorem \ref{fixedpoint}) in terms of the observable functions in Theorem \ref{theorem-mixing-erighted-observables}, which is crucial in deriving the martingale approximation. Proposition \ref{stepwise-centered-corrector-mapping} gives the definition of the corrector and its properties that will be used to construct the martingale approximation as given in \eqref{discrete-martingale-approximation}. 
\begin{proposition}\label{properties-of-function-spaces}
For $\eta\in(0, \eta_0/2]$, and any $0<\g \leq 1$, $P_t$ maps $C^{\gamma}_{\eta, H}(H\times\To^n)$ into $C^{\g}_{2\eta, H}(H\times\To^n)$;  If we further assume that $\Psi\in C^{\g}(\To^n, H)$, then $P_t$ maps $C^{\gamma}_{\eta, H}(H\times\To^n)\cap C_{\eta, \To^n}^{\g}(H\times\To^n)$ into $C^{\g}_{2\eta, H}(H\times\To^n)\cap C_{2\eta, \To^n}^{\g}(H\times\To^n)$. \\
\end{proposition}
\begin{proof}
Let $\phi\in C^{\gamma}_{\eta, H}(H\times\To^n)$. It follows from Lemma \ref{bounds} in the Appendix  that for $\eta\in(0, \eta_0/2]$, 
\begin{align*}
&|P_t\phi(w_1, h) - P_t\phi(w_2, h)| \\
&= \left|\mathbf{E}\phi(\Phi_{0, t, h}(w_1), \b_th) - \mathbf{E}\phi(\Phi_{0, t, h}(w_2), \b_th)\right| \leq \mathbf{E}\left|\phi(\Phi_{0, t, h}(w_1), \b_th)- \phi(\Phi_{0, t, h}(w_2), \b_th)\right|\\
&\leq \|\phi\|_{\g,\eta, H}\mathbf{E}\|\Phi_{0, t, h}(w_1) - \Phi_{0, t, h}(w_2)\|^{\g}\left(e^{\eta\|\Phi_{0, t, h}(w_1)\|^2}+e^{\eta\|\Phi_{0, t, h}(w_2)\|^2} \right)\\
&\leq C\|\phi\|_{\g,\eta, H}\left(\mathbf{E}\|\Phi_{0, t, h}(w_1) - \Phi_{0, t, h}(w_2)\|^{2}\right)^{\frac{\g}{2}}\left(\mathbf{E}\left(e^{2\eta\|\Phi_{0, t, h}(w_1)\|^2}+e^{2\eta\|w_{0, t, h}(w_2)\|^2} \right)\right)^{\frac12}\\
&\leq C\|\phi\|_{\g,\eta, H}\|w_1-w_2\|^{\g}e^{\frac{\g r_0}{2}t}e^{\frac{\g}{2}\eta\|w_1\|^2}\left(e^{\eta\|w_1\|^2}+e^{\eta\|w_2\|^2} \right),
\end{align*}
where $r_0$ is the constant from \eqref{continuous-initial-condition}. Hence we have 
\begin{align}\label{action-of-semigroup-on-weighted-holder}
|P_t\phi(w_1, h) - P_t\phi(w_2, h)|\leq C\|\phi\|_{\g,\eta, H}\|w_1-w_2\|^{\g}e^{\frac{\g r_0}{2}t}\left(e^{2\eta\|w_1\|^2}+e^{2\eta\|w_2\|^2} \right),
\end{align}
which shows that $P_t$ maps $C^{\gamma}_{\eta, H}(H\times\To^n)$ into $C^{\g}_{2\eta, H}(H\times\To^n)$. 

\vskip0.05in

Now assume $\Psi\in C^{\g}(\To^n, H)$ and let $\phi\in C^{\gamma}_{\eta, H}(H\times\To^n)\cap C_{\eta, \To^n}^{\g}(H\times\To^n)$. Then 
\begin{align*}
&|P_t\phi(w, h_1) -P_t\phi(w, h_2)|= \left|\mathbf{E}\phi(\Phi_{0, t, h_1}(w), \b_th_1) - \mathbf{E}\phi(\Phi_{0, t, h_2}(w), \b_th_2)\right|\\
&\leq \mathbf{E}\left|\phi(\Phi_{0, t, h_1}(w), \b_th_1)-\phi(\Phi_{0, t, h_1}(w), \b_th_2)\right| + \mathbf{E}\left|\phi(\Phi_{0, t, h_1}(w), \b_th_2) - \phi(\Phi_{0, t, h_2}(w), \b_th_2)\right|\\
&\leq \|\phi\|_{\g,\eta,\To^n}|h_1-h_2|^{\g}\mathbf{E}e^{\eta\|\Phi_{0, t, h_1}(w)\|^2} + \|\phi\|_{\g, \eta, H}\mathbf{E}\|\Phi_{0, t, h_1}(w) - \Phi_{0, t, h_2}(w)\|^{\g}\left(e^{\eta\|\Phi_{0, t, h_1}(w)\|^2}+e^{\eta\|\Phi_{0, t, h_2}(w)\|^2} \right)
\end{align*}
It follows from \eqref{eq: est1} and \eqref{continuousonhull} that 
\begin{align*}
&\mathbf{E}\|\Phi_{0, t, h_1}(w) - \Phi_{0, t, h_2}(w)\|^{\g}\left(e^{\eta\|\Phi_{0, t, h_1}(w)\|^2}+e^{\eta\|\Phi_{0, t, h_2}(w)\|^2} \right)\\
&\leq \left(\mathbf{E}\|\Phi_{0, t, h_1}(w) - \Phi_{0, t, h_2}(w)\|^{2}\right)^{\frac{\g}{2}}\left(\mathbf{E}\left(e^{2\eta\|\Phi_{0, t, h_1}(w)\|^2}+e^{2\eta\|\Phi_{0, t, h_2}(w)\|^2} \right)\right)^{\frac12}\\
&\leq Ce^{\frac{r_0\g}{2}t}e^{2\eta\|w\|^2}\|\Psi\|_{\g}^{\g}|h_1 -h_2|^{\g}.
\end{align*}
Hence 
\begin{align}\label{action-of-semigroup-on-weighted-holder-torus}
|P_t\phi(w, h_1) -P_t\phi(w, h_2)|\leq C(\|\phi\|_{\g,\eta,\To^n}+ \|\phi\|_{\g, \eta, H}\|\Psi\|_{\g}^{\g})e^{\frac{r_0\g}{2}t}e^{2\eta\|w\|^2}|h_1 -h_2|^{\g},
\end{align}
which shows that $P_t$ maps $C^{\gamma}_{\eta, H}(H\times\To^n)\cap C_{\eta, \To^n}^{\g}(H\times\To^n)$ into $C^{\g}_{2\eta, H}(H\times\To^n)\cap C_{2\eta, \To^n}^{\g}(H\times\To^n)$.
\end{proof}

The following theorem shows that the homogenized process is mixing over the family of observables normalized by the quasi-periodic invariant measure as in \eqref{widetilde-phi}, although it is not mixing in the usual sense (by centering the observables with the unique invariant measure of the homogenized process). Recall that $\r$ \eqref{rho-introduction} is the distance weighted by $e^{\eta\|w\|^2}$ and therefore depends on $\eta$. We will use the following function space as an auxiliary tool. Let  
\[[\phi]_{\g, \eta, H_{\r}} = \sup_{\substack{h\in\To^n\\ 0<\r(w_1, w_2)\leq 1}} \frac{|\phi(w_1, h)-\phi(w_2, h)|}{\r(w_1, w_2)^{\g}}, \]
be the Hölder semi-norm under the metric $\r$,  and set $C^{\g}_{\eta, H_{\r}}(H\times\To^n)$ as the space of bounded Hölder continuous functions 
\begin{align}\label{holder}
C^{\g}_{\eta, H_{\r}}(H\times\To^n) = \left\{\phi\in C(H\times\To^n): \|\phi\|_{\g,\eta, H_{\r}}: = \sup_{(w, h)\in H\times\To^n}|\phi(w, h)| +[\phi]_{\g, \eta, H_{\r}} <\infty\right\}.
\end{align}

\begin{theorem}\label{theorem-mixing-erighted-observables}
For any $\g\in (0, 1]$, $\eta\in (0, \eta_0/2]$ and $\phi\in C^{\g}_{\eta, H}(H\times\To^n) $ we have 
\begin{align}\label{mixing-erighted-observables}
|P_t\widetilde{\phi}(w, h)|\leq C\|\phi\|_{\g,\eta, H}e^{2\eta\|w\|^2}e^{-\L t}, \quad \forall t\geq 0, (w, h)\in H\times\To^n,
\end{align}
where  $\L = \frac{\g\varpi}{5}$, the mixing rate $\varpi$ is from Theorem \ref{fixedpoint}, and $C$ is a positive constant independent of $(w, h)$. 
\end{theorem}
\begin{proof}
We first use a H\"older cutoff function to decompose the function into a bounded part and an unbounded part. Once we have estimates on these two parts, then inequality  \eqref{mixing-erighted-observables} follows. The proof is divided into three steps. Assume without loss of generality that  $\|\phi\|_{\g,\eta, H} = 1$. 

\vskip0.05in

{\it Step 1: Cutoff and splitting.} For any $R>0$, let $\chi_{R}:H\ra\R\in C_{\eta, H}^{\g}(H\times \To^n)$ satisfying $0\leq\chi_{R}\leq 1$, with $\chi_R(w) = 1$ for $\|w\|\leq R$ and $\chi_{R}(w) = 0$ for $\|w\|\geq R+1$. We can actually choose a $\chi_{R}$ such that $\|\chi_R\|_{\g,\eta, H}\leq 2$. Also denote by $\overline{\chi}_{R} = 1-\chi_{R}$. Then 
\begin{align}\notag
|P_t\widetilde{\phi}(w, h)| &= |P_t\left(\chi_R\phi+\overline{\chi}_{R}\phi\right)(w, h) - \langle \G_{\b_th}, (\chi_R\phi)(\cdot, \b_th) + (\overline{\chi}_{R}\phi)(\cdot, \b_th)\rangle|\\\notag
&\leq |P_t(\chi_R\phi)(w, h) - \langle \G_{\b_th}, (\chi_R\phi)(\cdot, \b_th)\rangle| + |P_t(\overline{\chi}_{R}\phi)(w, h) - \langle \G_{\b_th},  (\overline{\chi}_{R}\phi)(\cdot, \b_th)\rangle|\\\label{mixing-weighted-holder-estimate}
&:= I_1 + I_2. 
\end{align}
We claim that $\chi_R\phi\in C^{\g}_{\eta, H_{\r}}(H\times\To^n)$ which will be used in estimating the bounded part. Indeed, since $\chi_R\phi$ vanishes outside of the ball $\|w\|\leq R+1$, in view of the definition \eqref{weighted-holder-norm}, one has  
\begin{align}\label{chi-phi-finite-ball}
\sup_{(w, h)\in H\times\To^n}|\chi_R(w)\phi(w, h)| \leq \sup_{\substack{(w, h)\in H\times\To^n\\ \|w\|\leq R+1}}|\phi(w, h)| \leq \|\phi\|_{\g,\eta, H}e^{\eta(R+1)^2}.
\end{align}
Let 
\[S = \{(w_1, w_2)\in H^2: \|w_1\|\leq R+1, \|w_2\|\geq R+1, 0<\|w_1-w_2\|\leq 1\},\] 
and 
\[S_{\r} = \{(w_1, w_2)\in H^2: \|w_1\|\leq R+1, \|w_2\|\geq R+1, 0<\r(w_1, w_2)\leq 1\}.\]
It is clear that $S_{\r}\subset S$ since $\|w_1-w_2\|\leq \r(w_1, w_2)$. It follows from $\chi_{R}(w_2) = 0$, $\|w_2\|\leq R+2$ and \eqref{chi-phi-finite-ball} that 
\begin{align*}
\sup_{(w_1,w_2)\in S_{\r}}\frac{|\chi_{R}(w_1)\phi(w_1, h) - \chi_{R}(w_2)\phi(w_2, h)|}{\r(w_1, w_2)^{\g}} &= \sup_{(w_1,w_2)\in S_{\r}}\frac{|\chi_{R}(w_1)\phi(w_1, h) - \chi_{R}(w_2)\phi(w_1, h)|}{\r(w_1, w_2)^{\g}}\\
&\leq 2e^{\eta(R+1)^2}\sup_{(w_1,w_2)\in S}\frac{|\chi_{R}(w_1) - \chi_{R}(w_2)|}{\|w_1-w_2\|^{\g}}\\
&\leq 4e^{2\eta(R+2)^2}.
\end{align*}
Now let 
\[S^R = \{(w_1, w_2)\in H^2: \|w_1\|, \|w_2\|\leq R+1, 0<\|w_1-w_2\|\leq 1\},\] 
and 
\[S_{\r}^R = \{(w_1, w_2)\in H^2: \|w_1\|, \|w_2\|\leq R+1, 0<\r(w_1, w_2)\leq 1\}.\]
First note that 
for $\|w_1\|, \|w_2\|\leq R+1$, 
\begin{align*}
|\chi_R(w_1)\phi(w_1, h)-\chi_R(w_2)\phi(w_2, h)| \leq 2e^{\eta(R+1)^2}\Big(|\chi_R(w_1) - \chi_R(w_2) + |\phi(w_1, h)-\phi(w_2, h)|\Big)
\end{align*}
Hence 
\begin{align*}
&\sup_{(w_1, w_2)\in S_{\r}^R}\frac{|\chi_R(w_1)\phi(w_1, h)-\chi_R(w_2)\phi(w_2, h)|}{\r(w_1, w_2)^{\g}}\\
&\leq 2e^{\eta(R+1)^2}\Big(\sup_{(w_1, w_2)\in S^R}\frac{|\chi_R(w_1)-\chi_R(w_2)|}{\|w_1 - w_2\|^{\g}}+ \sup_{(w_1, w_2)\in S^R}\frac{|\phi(w_1, h)-\phi(w_2, h)|}{\|w_1 - w_2\|^{\g}}\Big)\\
&\leq 4e^{2\eta(R+2)^2}.
\end{align*}
It then follows that $\chi_R\phi\in C^{\g}_{\eta, H_{\r}}(H\times\To^n)$ and 
\[\|\chi_R\phi\|_{\g,\eta,H_{\r}}\leq Ce^{2\eta(R+2)^2},\]
where $C$ is a constant that does not depend on $R$.  The claim is proved. 

\vskip0.05in

{\it Step 2: Estimate the bounded part.} This is guaranteed by the exponential mixing in Theorem \ref{fixedpoint}.  Since Theorem \ref{fixedpoint} is given in terms of dual Lipschitz metric but here we are working with H\"older observables,  we first establish the fact that the dual Hölder metric on $\mathcal{P}(H)$ is bounded by the Wasserstein metric:
\begin{align}\label{dual-holder-by-rho}
\begin{split}
\sup_{\|\varphi\|_{\g,\eta,H_{\r}}\leq 1}|\langle\mu_1, \varphi\rangle -\langle\mu_2, \varphi\rangle|\leq 2(\r(\mu_1, \mu_2))^{\g}, \quad \forall \mu_1, \mu_2 \in\mathcal{P}(H),
\end{split}
\end{align}
 To prove, let $[\cdot]_{\g}$ be the usual global H\"older semi-norm given by 
\[[\varphi]_{\g}= \sup_{w_1\neq w_2}\frac{|\varphi(w_1) - \varphi(w_2)|}{\r(w_1, w_2)^{\g}}, \quad \varphi\in C(H).\] 
Recall the H\"older norm given in \eqref{holder}. A direct verification gives 
\begin{align*}
\{\varphi\in C(H): \|\varphi\|_{\g,\eta,H_{\r}}\leq 1\}\subset\{\varphi\in C(H): [\varphi]_{\g}\leq 2\}.
\end{align*}
Since $\r$ is a metric and $\g\in (0, 1]$, $\r^{\g}$ is still a metric. Therefore by the definition of Wasserstein metric \eqref{rho-on-PH}, the Monge-Kantorovich duality \eqref{Monge-Kantorovich duality}  and Jensen's inequality,  we have 
\begin{align*}
\begin{split}
\sup_{\|\varphi\|_{\g,\eta,H_{\r}}\leq 1}|\langle\mu_1, \varphi\rangle -\langle\mu_2, \varphi\rangle|&\leq \sup_{[\varphi]_{\g}\leq 2}|\langle\mu_1, \varphi\rangle -\langle\mu_2, \varphi\rangle|\leq 2\sup_{[\varphi]_{\g}\leq 1}|\langle\mu_1, \varphi\rangle -\langle\mu_2, \varphi\rangle|\\
& = 2\inf_{\mu\in\mathcal{C}(\mu_1, \mu_2)}\int_{H\times H} \rho^{\g}(u, v)\mu(dudv)\\
&\leq 2\inf_{\mu\in\mathcal{C}(\mu_1, \mu_2)}\left(\int_{H\times H} \rho(u, v)\mu(dudv)\right)^{\g}=2(\r(\mu_1, \mu_2))^{\g}.
\end{split}
\end{align*}
Hence \eqref{dual-holder-by-rho} is proved. Combining this fact with Theorem \ref{fixedpoint}, it follows that the first term in \eqref{mixing-weighted-holder-estimate} satisfies
\begin{align*}
I_1 &=  |P_t(\chi_R\phi)(w, h) - \langle \G_{\b_th}, (\chi_R\phi)(\cdot, \b_th)\rangle| \\
&=  |\langle\P_{0, t, h}^*\delta_{w}, (\chi_R\phi)(\cdot, \b_th)\rangle - \langle \P_{0, t, h}^*\G_{h}, (\chi_R\phi)(\cdot, \b_th)\rangle|\\
&\leq 2\|\chi_R\phi\|_{\g,\eta, H_{\r}}(\r(\P_{0,t, h}^*\delta_w, \P_{0, t, h}^*\G_{h}))^{\g}\leq Ce^{2\eta(R+2)^2}e^{-\g\varpi t}e^{2\eta\|w\|^2}. 
\end{align*}

\vskip0.05in

{\it Step 2: Estimate the unbounded part.} By Theorem \ref{fixedpoint}, the integral of the Lyapunov function by the invariant measure is finite uniformly in $h$, hence the invariant measure must have exponentially small tail. This can be used to control the unbounded part. Indeed, to estimate the second term in \eqref{mixing-weighted-holder-estimate}, observe that 
\begin{align*}
|P_t(\overline{\chi}_R\phi)(w, h)| 
&\leq \mathbf{E}_{(w, h)}(|\overline{\chi}_R\phi|)(X_t)\leq (\mathbf{E}_{(w, h)}\overline{\chi}_R(X_t))^{1/2}(\mathbf{E}_{(w, h)}|\phi(X_t)|^2)^{1/2}\\
&\leq \Big(\mathbf{P}\left(\|\Phi_{0, t, h}(w)\|\geq R\right)\Big)^{1/2} \Big(\mathbf{E}e^{2\eta\|\Phi_{0, t, h}(w)\|^2}\Big)^{1/2}\\
&\leq C e^{\eta\|w\|^2}e^{-\eta R^2}(\mathbf{E}e^{2\eta\|\Phi_{0, t, h}(w)\|^2})^{1/2}\leq C e^{2\eta\|w\|^2}e^{-\eta R^2},
\end{align*}
where we used Markov inequality, the fact that $\|\phi\|_{\g,\eta, H}=1$ and the estimate \eqref{eq: est1}. 

\vskip0.05in

In a similar fashion, note that 
\begin{align*}
 |\langle \G_{\b_th},  (\overline{\chi}_{R}\phi)(\cdot, \b_th)\rangle|
 &\leq \Big(\int_{H}\overline{\chi}_{R}(w)\G_{\b_th}(dw)\Big)^{1/2}\Big(\int_{H}\left|\phi(w, \b_th)\right|^2\G_{\b_th}(dw)\Big)^{1/2}\\
 &\leq \Big(\G_{\b_th}(\|w\|\geq R)\Big)^{1/2}\left(\int_{H}e^{2\eta\|w\|^2}\G_{\b_th}(dw)\right)^{1/2}\\
 &\leq Ce^{-\eta R^2}\left(\int_{H}e^{2\eta\|w\|^2}\G_{\b_th}(dw)\right)^{1/2}\leq Ce^{-\eta R^2}.
\end{align*}
Hence $I_2\leq C e^{2\eta\|w\|^2}e^{-\eta R^2}$. As a result, we have the following estimate on \eqref{mixing-weighted-holder-estimate}, 
\begin{align*}
|P_t\widetilde{\phi}(w, h)| \leq Ce^{\k\eta\|w\|^2}(e^{4\eta R^2-\g\varpi t} + e^{-\eta R^2}), 
\end{align*}
with constant $C$ independent of $R$. By choosing $R^2 = \frac{\g\varpi t}{5\eta}$, we obtain \eqref{mixing-erighted-observables} with $\L = \frac{\g\varpi}{5}$. The proof is complete.

\end{proof}

We now define the corrector and prove its regularity that will be used in the martingale approximation procedure. 

\begin{proposition}[The corrector]\label{stepwise-centered-corrector-mapping} 
For $\phi\in C^{\g}_{\eta, H}(H\times\To^n)$, define 
\[\chi(w, h) := \int_{0}^{\infty}P_t \widetilde{\phi}(w, h)dt = \int_{0}^{\infty}\mathcal{P}_{0, t, h}\phi(\cdot, \b_th)(w) - \langle\G_{\b_th},\phi(\cdot, \b_th) \rangle dt, \quad (w, h) \in H\times\To^n.\]
(1). (H\"older regularity in $H$) For $\eta\in(0, \eta_0/2]$, $\g\in(0, 1]$,  the corrector $\chi\in C^{\g_0}_{2\eta, H}(H\times\To^n)$ as long as $\phi\in C^{\g}_{\eta, H}(H\times\To^n)$, where $\g_0 = \frac{\L\g}{\L+\g r_0}$, with $\L$ from Theorem \ref{theorem-mixing-erighted-observables} and $r_0 $ is the constant from Appendix.  \\
(2). (H\"older regularity in $H$ and $\To^n$) If we assume $\Psi\in C^{\g}(\To^n, H)$ and $\phi\in C^{\g}_{\eta, H}(H\times\To^n)\cap C^{\g}_{\eta, \To^n}(H\times\To^n)$, then the associated corrector $\chi\in C^{\g_0}_{2\eta, H}(H\times\To^n)\cap C^{\g_1}_{2\eta, \To^n}(H\times\To^n)$, where  $\g_1 = \frac{\z}{5}\g_0$ with $\g_0$ as above and $\z$ from Theorem \ref{fixedpoint}.  In particular, we have 
\[\chi\in C^{\overline{\g}}_{2\eta, H}(H\times\To^n)\cap C^{\overline{\g}}_{2\eta, \To^n}(H\times\To^n), \]
where $\overline{\g} := \left(\frac{\g\varpi}{5(r_0+\varpi)}\right)^2$ and $\varpi$ is the mixing rate from Theorem \ref{fixedpoint}. 
\end{proposition}
\begin{proof}
The function $\chi$ is well defined in view of Theorem \ref{theorem-mixing-erighted-observables}.  The proof is divided in to two steps. 

\vskip0.05in

{\it Step 1: Proof of item (1).} The H\"older regularity in $H$ is preserved at finite time scales in view of \eqref{action-of-semigroup-on-weighted-holder}. The infinite time scale can be controlled through the exponential mixing \eqref{theorem-mixing-erighted-observables}. More specifically,  let $\phi\in C^{\g}_{\eta, H}(H\times\To^n)$. It follows from Theorem \ref{theorem-mixing-erighted-observables} and inequality \eqref{action-of-semigroup-on-weighted-holder} that, for $\eta\in(0, \eta_0/2]$,
\begin{align*}
|\chi(w_1, h)-\chi(w_2, h)|&\leq \int_0^T|P_t\widetilde{\phi}(w_1, h) -P_t\widetilde{\phi}(w_2, h)|dt + \int_{T}^{\infty}|P_t\widetilde{\phi}(w_1, h)| + |P_t\widetilde{\phi}(w_2, h)|dt\\
&\leq \int_0^T|P_t\phi(w_1, h) -P_t\phi(w_2, h)|dt + C\|\phi\|_{\g,\eta, H}e^{-\L T}\left(e^{2\eta\|w_1\|^2}+e^{2\eta\|w_2\|^2}\right)\\
&\leq C\|\phi\|_{\g,\eta, H}\left(\|w_1-w_2\|^{\g}e^{\g r_0T}+e^{-\L T}\right)\left(e^{2\eta\|w_1\|^2}+e^{2\eta\|w_2\|^2} \right).
\end{align*}
In view of Lemma \ref{choosing-Holder-exponent}, we have for any $0<\|w_1-w_2\|\leq 1$, 
\begin{align*}
|\chi(w_1, h)-\chi(w_2, h)|\leq C\|\phi\|_{\g,\eta, H}\left(e^{2\eta\|w_1\|^2}+e^{2\eta\|w_2\|^2} \right)\|w_1-w_2\|^{\g_0},
\end{align*}
with $\g_0 = \frac{\L\g}{\L+ \g r_0} = \frac{\g\varpi}{5r_0+\varpi}$. This also indicates that $\chi(w, h)$ is continuous in $w$ uniformly for $h$. The continuity of $\chi(w, h)$ in $h$ for fixed $w$ follows from the fact that $\chi_T(w, h): = \int_{0}^{T}P_t \widetilde{\phi}(w, h)dt$ is continuous in $h$ and $\chi_T(w, h)\ra \chi(w, h)$ uniformly for $h$. Thus $\chi\in C(H\times\To^n)$. It also follows from Theorem \ref{theorem-mixing-erighted-observables} that $|\chi(w, h)|\leq C\|\phi\|_{\g,\eta, H}e^{2\eta\|w\|^2}$. Hence $\chi\in C^{\g_0 }_{2\eta, H}(H\times\To^n)$ .

\vskip0.05in

{\it Step 2: Proof of item (2).} The idea is similar to the previous step. But the H\"older regularity in time symbols $\To^n$  at finite time scales needs more effort. Besides the H\"older estimate from \eqref{action-of-semigroup-on-weighted-holder-torus}, one has to deal with the interaction between the H\"older regularity of the invariant measure $\G$ from Theorem \ref{fixedpoint}, and that of the observable function $\phi$.  Since $\phi$ can be unbounded, we apply the cut-off idea as in the proof of Theorem \ref{theorem-mixing-erighted-observables}. 

To elaborate, let $\phi\in C^{\g}_{\eta, H}(H\times\To^n)\cap C^{\g}_{\eta, \To^n}(H\times\To^n)$ and assume $\Psi\in C^{\g}(\To^n, H)$. Then 
\begin{align*}
|\chi(w, h_1)- \chi(w, h_2)| &\leq \int_0^T|P_t\widetilde{\phi}(w, h_1) -P_t\widetilde{\phi}(w, h_2)|dt + \int_T^{\infty}|P_t\widetilde{\phi}(w, h_1) |+ |P_t\widetilde{\phi}(w, h_2) | dt \\
&:= I + II. 
\end{align*}
Note that by Theorem \ref{theorem-mixing-erighted-observables}, the second term can be estimated as
\begin{align}\label{Estimate-II}
II\leq C\|\phi\|_{\g,\eta, H}e^{2\eta\|w\|^2}e^{-\L T}. 
\end{align}
For the first term, observe that 
\begin{align}\label{Estimate-I}
\begin{split}
I &\leq \int_0^T|P_t\phi(w, h_1) -P_t\phi(w, h_2)|dt + \int_0^T\left|\langle\G_{\b_th_1}, \phi(\cdot, \b_th_1)\rangle - \langle\G_{\b_th_2}, \phi(\cdot, \b_th_2)\rangle\right| dt \\
&: = I_1 + I_2.
\end{split}
\end{align}
It then follows from the estimate \eqref{action-of-semigroup-on-weighted-holder-torus} that 
\begin{align}\label{Estimate-I1}
\begin{split}
I_1&\leq C(\|\phi\|_{\g,\eta,\To^n}+ \|\phi\|_{\g, \eta}\|\Psi\|_{\g}^{\g})e^{2\eta\|w\|^2}|h_1 -h_2|^{\g}\int_0^T e^{\frac{r_0\g}{2}t}dt \\
&\leq C(\|\phi\|_{\g,\eta,\To^n}+ \|\phi\|_{\g, \eta}\|\Psi\|_{\g}^{\g})e^{2\eta\|w\|^2}|h_1 -h_2|^{\g} e^{\g r_0T}. 
\end{split}
\end{align}
To estimate $I_2$, noting that as in the proof of Theorem \ref{theorem-mixing-erighted-observables}, especially by \eqref{dual-holder-by-rho} and Theorem \ref{fixedpoint} ($\z$-Hölder continuity of $\G$), we have 
\begin{align*}
I_{21}:=&\left|\langle\G_{\b_th_1}, \phi(\cdot, \b_th_1)\rangle - \langle\G_{\b_th_2}, \phi(\cdot, \b_th_1)\rangle\right| \\
&\leq \left|\langle\G_{\b_th_1}- \G_{\b_th_2}, \chi_R\phi(\cdot, \b_th_1)\rangle\right| +  \left|\langle\G_{\b_th_1}, (\overline{\chi}_R\phi)(\cdot, \b_th_1)\rangle\right|+ \left|\langle\G_{\b_th_2}, (\overline{\chi}_R\phi)(\cdot, \b_th_1)\rangle\right|\\
&\leq C\|\phi\|_{\g,\eta, H}e^{4\eta R^2}\|\G\|_{\z}^{\g}|h_1-h_2|^{\g\z}+ C\|\phi\|_{\g, \eta, H}e^{-\eta R^2}\\
&\leq C\|\phi\|_{\g,\eta, H}(\|\G\|_{\z}^{\g}+1)\left(e^{4\eta R^2}|h_1-h_2|^{\g\z}+ e^{-\eta R^2}\right)\leq C\|\phi\|_{\g,\eta, H}(\|\G\|_{\z}^{\g}+1)|h_1-h_2|^{\frac{\g\z}{5}}
\end{align*}
by Lemma \ref{choosing-Holder-exponent}. Also note that by Theorem \ref{fixedpoint}, 
\begin{align*}
I_{22}:=&\left|\langle\G_{\b_th_2}, \phi(\cdot, \b_th_1)\rangle - \langle\G_{\b_th_2}, \phi(\cdot, \b_th_2)\rangle\right|\leq \langle\G_{\b_th_2}, |\phi(\cdot, \b_th_1) - \phi(\cdot, \b_th_2)|\rangle\\
&\leq \|\phi\|_{\g,\eta,\To^n}|h_1-h_2|^{\g}\int_{H}e^{\eta\|w\|^2}\G_{\b_th_2}(dw)\leq C\|\phi\|_{\g,\eta,\To^n}|h_1-h_2|^{\g}.
\end{align*}
As a result, 
\begin{align*}
I_2\leq \int_0^T I_{21} + I_{22}dt \leq CT\left(\|\phi\|_{\g,\eta, H}(\|\G\|_{\z}^{\frac{1}{2-\g}}+1)+\|\phi\|_{\g,\eta,\To^n}\right)|h_1-h_2|^{\frac{\g\z}{5}}.
\end{align*}
Combining this with \eqref{Estimate-I}, \eqref{Estimate-I1} and \eqref{Estimate-II}, we have 
\begin{align*}
|\chi(w, h_1)- \chi(w, h_2)|& \leq C e^{2\eta\|w\|^2}\left(e^{\g r_0T}|h_1-h_2|^{\frac{\g\z}{5}}+ e^{-\L T}\right)\\
&\leq C e^{2\eta\|w\|^2}|h_1-h_2|^{\g_1}. 
\end{align*}
 by Lemma \ref{choosing-Holder-exponent} with $\g_1 = \frac{\z}{5}\g_0$. Hence $\chi\in C^{\g_1}_{2\eta, \To^n}(H\times\To^n)$. Note that 
 \[\min\{\g_0, \g_1\} = \g_1 = \frac{(\g\varpi)^2}{5(r_0+\varpi)(5r_0+\varpi)}.\]
So we can choose $\overline{\g}:= \left(\frac{\g\varpi}{5(r_0+\varpi)}\right)^2$ which is less than $\g_1$. This completes the proof of this proposition.

\end{proof}

We are now in a position to give the martingale approximation. For $T\geq 0$, let 
\begin{align}\label{discrete-martingale-approximation}
\int_0^T\widetilde{\phi}(X_t) dt = \int_0^N\widetilde{\phi}(X_t) dt + \int_N^T\widetilde{\phi}(X_t)dt= M_N+R_{N,T}, 
\end{align}
where $N$ is the integer part of $T$, 
\[M_N = \chi(X_N) - \chi(X_0) + \int_{0}^{N}\widetilde{\phi}(X_t)  dt\]
is the Dynkin martingale and 
\[R_{N,T} = - \chi(X_N) + \chi(X_0) + \int_N^T\widetilde{\phi}(X_t)dt\]
is the reminder term. 
Let $Z_N = M_N-M_{N-1}$ for $N\geq 1$ be the associated martingale difference. In what follows, we will show that $M_N$ is indeed a martingale and $R_{N,T}$ is a negligible term that vanishes as $T\ra \infty$.  Let $M_T = \chi(X_T) - \chi(X_0) + \int_{0}^{T}\widetilde{\phi}(X_t)  dt$ for $T\geq 0$.  
\begin{lemma}
$\{M_T\}_{T\geq 0}$ is a zero mean martingale w.r.t the filtration $\{\mathcal{F}_T\}$.
\end{lemma}
\begin{proof}
The martingale property follows from the Markov property of the homogenized process $X_t$ as follows. 
\begin{align*}
\mathbf{E}[M_T|\mathcal{F}_s] = \mathbf{E}[\chi(X_T)|\mathcal{F}_s] - \chi(X_0) +  \int_{0}^s\mathbf{E}[\widetilde{\phi}(X_u)|\mathcal{F}_s] du + \int_{s}^T\mathbf{E}[\widetilde{\phi}(X_u)|\mathcal{F}_s] du.
\end{align*}
Since $X_u$ is $\mathcal{F}_s$ measurable for $0\leq u\leq s$, it follows that 
\[\int_{0}^s\mathbf{E}[\widetilde{\phi}(X_u)|\mathcal{F}_s] du = \int_{0}^s\widetilde{\phi}(X_u) du.\]
Moreover, by the Markov property,
\begin{align*}
\int_{s}^T\mathbf{E}[\widetilde{\phi}(X_u)|\mathcal{F}_s] du = \int_s^{\infty}P_{u-s}\widetilde{\phi}(X_s) du -  \int_T^{\infty}P_{u-T}(P_{T-s}\widetilde{\phi})(X_s) du = \chi(X_s) - \mathbf{E}[\chi(X_T)|\mathcal{F}_s]. 
\end{align*}
Hence $\mathbf{E}[M_T|\mathcal{F}_s]  = M_s$.  

\vskip0.05in

It is zero mean since 
\begin{align*}
M_T& = \chi(X_T) - \chi(X_0) + \int_{0}^{T}\widetilde{\phi}(X_t)  dt \\
& = \int_{T}^{\infty}P_{t-T}\widetilde{\phi}(X_T)dt - \int_0^{\infty}P_t\widetilde{\phi}(X_0)dt + \int_{0}^{T}\widetilde{\phi}(X_t)  dt \\
& = \int_{T}^{\infty}\mathbf{E}_{(w, h)}\left[\widetilde{\phi}(X_t)|\mathcal{F}_T\right]dt- \int_0^{\infty}P_t\widetilde{\phi}(w, h)dt + \int_{0}^{T}\widetilde{\phi}(X_t)  dt,
\end{align*}
which implies $\mathbf{E}_{(w, h)}M_T$ = 0. 
\end{proof}
The following lemma gives estimates on the moments of the martingale $M_N$ and its associated martingale difference. 
\begin{lemma}[Bounds on the martingale]\label{martingalebounds}
For integer $p\geq 1$, $\eta\in(0, 2^{-p-1}\eta_0]$ and $\phi\in C^{\g}_{\eta, H}(H\times\To^n)$, one has 
\[\mathbf{E}_{(w, h)}|M_T|^{2^p}\leq C(T^{2-2^{-p}}+1)e^{2^{p+1}\eta\|w\|^2}, \quad \mathbf{E}_{(w, h)}|Z_N|^{2^p}\leq Ce^{2^{p+1}\eta\|w\|^2},\]
for $T\geq 0$ and $N\geq 1$. Also with a larger constant $C$, 
\[P_t\mathbf{E}_{(w, h)}|M_T|^{2p}\leq C(T^{2-2^{-p}}+1)e^{2^{p+1}\eta\|w\|^2},\quad \forall t\geq 0.\]
\end{lemma}
\begin{proof}
By Proposition \ref{stepwise-centered-corrector-mapping}, we know that $\chi\in C^{\g_0}_{2\eta, H}(H\times\To^n)$. Hence  $\chi^{2^p}\in C^{\g_0}_{2^{p+1}\eta, H}(H\times\To^n)$. Besides, since $\phi\in C^{\g}_{\eta, H}(H\times\To^n)$, $|\widetilde{\phi}|^{2^p}\in C^{\g}_{2^{p+1}\eta, H}(H\times\To^n)$. It follows from estimate \eqref{eq: est1} that for $\eta\in(0, 2^{-p-1}\eta_0]$, and any $t\geq 0$, 
\begin{align*}
\mathbf{E}_{(w, h)}|M_T|^{2^p}&\leq C\left(\mathbf{E}_{(w, h)}|\chi(X_T)|^{2^p}+|\chi(w, h)|^{2^p}+ T^{1-2^{-p}}\int_0^T\mathbf{E}_{(w, h)}|\widetilde{\phi}(X_t)|^{2^p}dt\right)\\
&\leq  C\left(\mathbf{E}e^{2^{p+1}\eta\|\Phi_{0, T, h}(w)\|^2} +e^{2^{p+1}\eta\|w\|^2}+ T^{1-2^{-p}}\int_0^T\mathbf{E}e^{2^{p+1}\eta\|\Phi_{0, t, h}(w)\|^2}dt\right)\\
&\leq C(T^{2-2^{-p}}+1)e^{2^{p+1}\eta\|w\|^2}. 
\end{align*}
Similarly, one can show that for any $N\geq 1$, 
\begin{align*}
\mathbf{E}_{(w, h)}|Z_N|^{2^p}\leq Ce^{2^{p+1}\eta\|w\|^2},
\end{align*}
where $C$ does not depend on $N, h$. It follows from \eqref{eq: est1} that 
\[P_t\mathbf{E}_{(w, h)}|M_T|^{2^p}\leq C(T^{2-2^{-p}}+1)\mathbf{E}e^{2^{p+1}\eta\|\Phi_{0, t, h}(w)\|^2}\leq C(T^{2-2^{-p}}+1)e^{2^{p+1}\eta\|w\|^2}. \]
\end{proof}

The following lemma gives estimates on the remainder term. 
\begin{lemma}\label{the-error-term}
Let $R_{N, T}$ be as in \eqref{discrete-martingale-approximation}. Then for any initial condition $X_0 = (w, h)$, $\eta\in(0, 2^{-4}\eta_0]$ and $\phi\in C^{\g}_{\eta, H}(H\times\To^n)$, 
\begin{align}\label{convergence-remainder}
\lim_{T\ra\infty}\frac{1}{\sqrt{T}}R_{N, T} = 0, \quad \mathbf{P}\text{-}\mathrm{a.s}..
\end{align}
\end{lemma}
\begin{proof}
Since $N$ is the integer part of $T$, it suffices to show 
\begin{align*}
\lim_{N\ra\infty}\frac{1}{\sqrt{N}}\sup_{N\leq t\leq N+1}R_{N, t} = 0, \quad \mathbf{P}\text{-}\mathrm{a.s}..
\end{align*}
By Proposition \ref{stepwise-centered-corrector-mapping}, we have that 
\[|\chi(X_N)| \leq Ce^{2\eta\|\Phi_{0, N, h}(w)\|^2}.\]
Since $\phi\in C^{\g}_{\eta, H}(H\times\To^n)$,  it also holds that 
\begin{align}
\sup_{N\leq t\leq N+1}\left|\int_N^t\widetilde{\phi}(X_s)ds\right|\leq C\sup_{N\leq t\leq N+1}e^{2\eta\|\Phi_{0, t, h}(w)\|^2}.
\end{align}
It then follows from the Markov inequality, estimates \eqref{enstrophy} and \eqref{eq: est1} that for any $K>0$, 
\begin{align*}
\mathbf{P}\left(\sup_{N\leq t\leq N+1}e^{2\eta\|\Phi_{0, t, h}(w)\|^2}>K\right)\leq Ce^{2^4\eta\|w\|^2}K^{-8}.
\end{align*}
Hence
\begin{align*}
&\sum_{N=1}^{\infty}\mathbf{P}\left(\sup_{N\leq t\leq N+1}\left(|\chi(X_N)| +\chi(w, h) + \left|\int_N^t\widetilde{\phi}(X_s)ds\right|\right)\geq N^{\frac14} \right)\\
&\leq \sum_{N=1}^{\infty}\mathbf{P}\left(C\sup_{N\leq t\leq N+1}e^{2\eta\|\Phi_{0, t, h}(w)\|^2}\geq N^{\frac14}\right)\leq Ce^{2^4\eta\|w\|^2} \sum_{N=1}^{\infty}N^{-2}<\infty,
\end{align*}
By the Borel-Cantelli lemma, there is an almost surely finite random integer time $N_0(\o)$ such that for $N>N_0(\o)$, 
\begin{align}\label{convergence-rate-error}
\sup_{N\leq t\leq N+1}R_{N, t}\leq N^{1/4}, 
\end{align}
which implies \eqref{convergence-remainder}. 
\end{proof}
\subsection{The asymptotic variance}\label{subsection-asymptotic-variance}
This subsection is devoted to the proof of the existence of the asymptotic variance and its properties. 
\begin{proposition}[The asymptotic variance]
For any $(w, h)\in H\times\To^n$, $\eta\in(0,\eta_0/16]$ and  $\phi\in C_{\eta, H}^{\g}(H\times\To^n)$, we have 
\begin{align}\label{eq-the-asymptotic-variance}
\lim_{T\ra\infty}\frac{1}{T}\mathbf{E}\left(\int_0^T\widetilde{\phi}\left(X_t(w,h)\right)dt\right)^2 = 2\int_{H\times\To^n}\widetilde{\phi}(w, h)\chi(w, h)\G_{h}(dw)\lambda(dh): =\sigma_{\phi}^2. 
\end{align}
\end{proposition}
\begin{proof}
By the Markov property of the homogenized process, one has 
\begin{align*}
\frac{1}{T}\mathbf{E}\left(\int_0^T\widetilde{\phi}\left(X_t(w, h)\right)dt\right)^2 
&= \frac{2}{T}\mathbf{E}\int_0^T\int_s^T\widetilde{\phi}\left(X_t(w, h)\right)\widetilde{\phi}\left(X_s(w, h)\right)dtds\\
& = \frac{2}{T}\int_0^T\mathbf{E}\left[\widetilde{\phi}\left(X_s(w, h)\right)\int_s^T\mathbf{E}\left[\widetilde{\phi}\left(X_t(w, h)\right)|\mathcal{F}_s\right]dt\right]ds\\
& = \frac{2}{T}\int_0^T\left\langle P_{s}^*\d_{(w, h)}, \widetilde{\phi}\int_{0}^{T-s}P_{t}\widetilde{\phi}dt \right\rangle ds. 
\end{align*}
In view of the weak convergence in Proposition \ref{quasi-periodic-convergence-of-discrete-average} and the definition of the corrector in Proposition \ref{stepwise-centered-corrector-mapping}, we expect that as $T\ra \infty$,
\begin{align}\notag
\Big| \frac{2}{T}\int_0^T\left\langle P_{s}^*\d_{(w, h)}, \widetilde{\phi}\int_{0}^{T-s}P_{t}\widetilde{\phi}dt \right\rangle ds &-  \frac{2}{T}\int_0^T\left\langle P_{s}^*\d_{(w, h)}, \widetilde{\phi}\chi\right\rangle ds\Big| \\\label{eq-the-asymptotic-variance-1}
&= \left|\frac{2}{T}\int_0^T\left\langle P_{s}^*\d_{(w, h)}, \widetilde{\phi}(\chi-\chi_{T-s})\right\rangle ds\right|\ra 0,
\end{align}
where $\chi_{T-s} = \int_{0}^{T-s}P_{t}\widetilde{\phi}dt $.  Indeed, it follows from Theorem that 
\[|\chi-\chi_{T-s}|\leq \int_{T-s}^{\infty}\left|P_t\widetilde{\phi}(w, h)\right|dt\leq Ce^{2\eta\|w\|^2}e^{-\L(T-s)}.\]
Since $\widetilde{\phi}\in C_{\eta, H}^{\g}(H\times\To^n)$, and $e^{2\eta\|\cdot\|^2}\in C_{4\eta, H}^{\g}(H\times\To^n)$, it follows that $\widetilde{\phi}e^{2\eta\|\cdot\|^2}\in C_{8\eta, H}^{\g}(H\times\To^n)$. Hence by estimate \eqref{eq: est1}, we have for $\eta\in(0,\eta_0/8]$, and any $s\geq 0$, 
\begin{align*}
\left\langle P_s^*\d_{(w, h)}, |\widetilde{\phi}|e^{2\eta\|\cdot\|^2}\right\rangle\leq C\left\langle P_s^*\d_{(w, h)}, e^{8\eta\|\cdot\|^2}\right\rangle =C \mathbf{E}e^{8\eta\|\Phi_{0, s, h}(w)\|^2}\leq Ce^{8\eta\|w\|^2}. 
\end{align*}
Hence 
\begin{align*}
\left|\frac{1}{T}\int_0^T\left\langle P_{s}^*\d_{(w, h)}, \widetilde{\phi}(\chi-\chi_{T-s})\right\rangle ds\right|\leq Ce^{8\eta\|w\|^2}\frac{1}{T}\int_0^Te^{-\L(T-s)}ds\ra 0, 
\end{align*}
which implies the limit \eqref{eq-the-asymptotic-variance-1}.

\vskip0.05in

By Proposition \ref{stepwise-centered-corrector-mapping}, $\chi\in C^{\g_0}_{2\eta, H}(H\times\To^n)$. Hence 
$|\widetilde{\phi}\chi|^2\in C^{\g_0}_{8\eta, H}(H\times\To^n)$. Then by estimate \eqref{eq: est1}, for $\eta\in(0, \eta_0/8]$,
\begin{align*}
\limsup_{T\ra\infty}\frac{1}{T}\int_0^T\left\langle P_{s}^*\d_{(w, h)}, \left|\widetilde{\phi}\chi\right|^{2}\right\rangle ds\leq C\limsup_{T\ra\infty}\frac{1}{T}\int_0^T\mathbf{E}e^{8\eta\|\Phi_{0, s, h}(w)\|^2}ds<\infty.
\end{align*}
Combining this moment bound with the weak convergence in Proposition  \ref{quasi-periodic-convergence-of-discrete-average}, we obtain the desired convergence
\begin{align*}
\lim_{T\ra\infty}\frac{2}{T}\int_0^T\left\langle P_{s}^*\d_{(w, h)}, \widetilde{\phi}\chi\right\rangle ds = 2\int_{H\times\To^n}\widetilde{\phi}(w, h)\chi(w, h)\G_{h}(dw)\lambda(dh), 
\end{align*}
which combined with \eqref{eq-the-asymptotic-variance-1} implies the desired \eqref{eq-the-asymptotic-variance}. 
\end{proof}

We will omit the $\phi$ in $\sigma_{\phi}^2$ just for notational convenience. The following proposition gives further properties related to the asymptotic variance. In particular, the Hölder regularity of the particular observable function $F$ as below, plays an important role when estimating the rate of convergence in the central limit theorem.  
\begin{proposition}\label{Holder-regularity-of-Y}
For $\g\in(0, 1]$, $\eta\in (0, 2^{-5}\eta_0]$, $\phi\in C^{\g}_{\eta, H}(H\times\To^n)\cap C_{\eta, \To^n}^{\g}(H\times\To^n)$, and $X_0 = (w, h)$, let 
\[Y(w, h) = \mathbf{E}_{(w, h)}M_1^2= \mathbf{E}_{(w, h)}\left(\chi(X_1)-\chi(X_0)+ \int_{0}^1\widetilde{\phi}(X_t)dt\right)^2.\]
Assume $\Psi\in C^{\g}(\To^n, H)$. Then $Y \in C^{\overline{\g}}_{2^5\eta, H}(H\times\To^n)\cap C^{\overline{\g}}_{2^5\eta, \To^n}(H\times\To^n)$ with $\overline{\g}$  from Proposition \ref{stepwise-centered-corrector-mapping}. Furthermore, the function 
\[F(h):=\int_{H}Y(w, h)\G_h(dw) = \langle \G_h, Y(\cdot, h)\rangle \]
is in $C^{\overline{\g}_0}(\To^n, \R)$ with $\overline{\g}_0 = \left(\frac{\g\varpi}{5(r_0+\varpi)}\right)^3$. We also have 
\begin{align}\label{a-formula-for-sigma^2}
\s^2 = \int_{\To^n} F(h)\lambda(dh) = \int_{H\times\To^n}Y(w, h)\G_h(dw)\lambda(dh).
\end{align}

\end{proposition}
\begin{proof}
The proof is divided into three steps. 

\vskip0.05in

{\it Step 1: H\"older regularity of $Y$.} This follows from Proposition \ref{properties-of-function-spaces} and \ref{stepwise-centered-corrector-mapping}, and  the following representation by  the Markov property (see \cite{Shi06,KS12} for the homogeneous case): 
\begin{align}\label{decomposition-of-Y}
Y(w, h) = \chi^2(w, h) &+ P_1\chi^2(w, h) - 2\chi(w, h)P_1\chi(w, h) + 2\int_0^1P_t(\widetilde{\phi}P_{1-t}\chi)(w, h)dt \\\notag
&- 2\chi(w, h)\int_0^1P_t\widetilde{\phi}(w, h)dt + 2\int_0^1\int_0^tP_{\tau}\left(\widetilde{\phi}P_{t-\tau}\widetilde{\phi}\right)(w, h)d\tau dt.
\end{align}
By Proposition \ref{stepwise-centered-corrector-mapping} we know that $\chi\in C^{\overline{\g}}_{2\eta, H}(H\times\To^n)\cap C^{\overline{\g}}_{2\eta, \To^n}(H\times\To^n)$. Hence for $\eta\in(0, \eta_0/2]$,
\[\chi^2\in C^{\overline{\g}}_{4\eta, H}(H\times\To^n)\cap C^{\overline{\g}}_{4\eta, \To^n}(H\times\To^n).\] 
Since $\overline{\g}< \g$, it follows from Proposition \ref{properties-of-function-spaces} that for $\eta\in(0, 2^{-3}\eta_0]$, one has 
\[P_1\chi \in C^{\overline{\g}}_{4\eta, H}(H\times\To^n)\cap C^{\overline{\g}}_{4\eta, \To^n}(H\times\To^n),\]
and for $t\in [0, 1]$, 
\[\widetilde{\phi}P_{1-t}\chi, \chi P_1\chi, P_1\chi^2 \in C^{\overline{\g}}_{2^3\eta, H}(H\times\To^n)\cap C^{\overline{\g}}_{2^3\eta, \To^n}(H\times\To^n).\]
It then follows from \eqref{action-of-semigroup-on-weighted-holder} and \eqref{action-of-semigroup-on-weighted-holder-torus} that for $\eta\in (0, 2^{-4}\eta_0]$, 
\[\int_0^1P_t(\widetilde{\phi}P_{1-t}\chi)(w, h)dt \in C^{\overline{\g}}_{2^4\eta, H}(H\times\To^n)\cap C^{\overline{\g}}_{2^4\eta, \To^n}(H\times\To^n).\]
In a similar way, one can deduce that the remaining two integrals in \eqref{decomposition-of-Y} also belongs to the same function space. This shows that $Y \in C^{\overline{\g}}_{2^4\eta, H}(H\times\To^n)\cap C^{\overline{\g}}_{2^4\eta, \To^n}(H\times\To^n)$. 

\vskip0.05in

{\it Step 2: H\"older regularity of $F$. } The proof is similar to the estimate of $I_2$ in second step of the proof for Proposition \ref{stepwise-centered-corrector-mapping}.  Note that 
\begin{align*}
|F(h_1)-F(h_2)| &\leq \left|\langle \G_{h_1}, Y(\cdot, h_1)\rangle  - \langle \G_{h_2}, Y(\cdot, h_1)\rangle\right| + \left|\langle \G_{h_2}, Y(\cdot, h_1)\rangle- \langle \G_{h_2}, Y(\cdot, h_2)\rangle\right|\\
&:= I_1 + I_2.
\end{align*}
And using the same functions $\chi_{R}, \overline{\chi}_R$ as in the proof of Theorem \ref{theorem-mixing-erighted-observables}, together with the Hölder continuity of $\G_h$ and the fact \eqref{dual-holder-by-rho}, we have for $\eta\in (0, 2^{-5}\eta_0]$,
\begin{align*}
I_1&\leq \left|\langle \G_{h_1}-\G_{h_2}, (\chi_RY)(\cdot, h_1)\rangle\right|+\left|\langle \G_{h_1}-  \G_{h_2}, (\overline{\chi}_RY)(\cdot, h_1)\rangle\right|\\
&\leq C\|Y\|_{\overline{\g}, 2^4\eta, H}e^{2^{6}\eta R^2}(\r(\G_{h_1}, \G_{h_2}))^{\overline{\g}}+ C\|Y\|_{\overline{\g}, 2^4\eta, H}e^{-2^{4}\eta R^2}\\
&\leq  C\|Y\|_{\overline{\g}, 2^4\eta, H}\left(e^{2^{6}\eta R^2}|h_1-h_2|^{\overline{\g}\z}+ e^{-2^{4}\eta R^2}\right),
\end{align*}
where we used the uniform integrability $\int_{H}e^{2^5\eta\|w\|^2}\G_{h}(dw)\leq C$ in the second inequality, which is a consequence of Theorem \ref{fixedpoint} by taking $\k = 2^4$. It then follows from Lemma \ref {choosing-Holder-exponent} that 
\[I_1\leq C\|Y\|_{\overline{\g}, 2^4\eta, H}|h_1- h_2|^{\frac{\z\overline{\g}}{5}}.\]
Also note that 
\begin{align*}
I_2&\leq \langle \G_{h_2}, \left|Y(\cdot, h_1)- Y(\cdot, h_2)\right|\rangle\\
&\leq \|Y\|_{\overline{\g}, 2^4\eta, \To^n}|h_1-h_2|^{\overline{\g}}\int_{H}e^{2^4\eta\|w\|^2}\G_{h_2}(dw)\leq C \|Y\|_{\overline{\g}, 2^4\eta, \To^n}|h_1-h_2|^{\overline{\g}}.
\end{align*}
Since $\frac{\z\overline{\g}}{5}\leq\overline{\g}$, we deduce that 
\[|F(h_1)-F(h_2)| \leq C\left(\|Y\|_{\overline{\g}, 2^4\eta, \To^n}+ \|Y\|_{\overline{\g}, 2^4\eta, H}\right)|h_1-h_2|^{\frac{\z\overline{\g}}{5}}, \quad\forall h_1, h_2\in \To^n.\]
Hence $F\in C^{\overline{\g}_0}(\To^n, \R)$ with $\overline{\g}_0 = \frac{\z\overline{\g}}{5}= \left(\frac{\g\varpi}{5(r_0+\varpi)}\right)^3$. 

\vskip0.05in

{\it Step 3: The representation  \eqref{a-formula-for-sigma^2} of asymptotic variance. } This follows from the invariance property of the invariant measure $\G_h(dw)\lambda(dh)$ and the representation \eqref{decomposition-of-Y}. Indeed, letting $m(dwdh) = \G_h(dw)\lambda(dh)$ and $\chi_t = \int_0^{t}P_r\widetilde{\phi}dr$, then by the invariance of $m$ under $P_t$, one has 
\begin{align*}
\int_0^1\int_0^t \left\langle m, P_{\tau}\left(\widetilde{\phi}P_{t-\tau}\widetilde{\phi}\right)\right\rangle d\tau dt = \int_0^1\int_0^t \left\langle m, \widetilde{\phi}P_{\tau}\widetilde{\phi}\right\rangle d\tau dt = \int_0^1\left\langle m, \widetilde{\phi}\chi_t\right\rangle dt,\\
\int_0^1 \left\langle m, P_t(\widetilde{\phi}P_{1-t}\chi) \right\rangle dt = \int_0^1 \left\langle m, \widetilde{\phi}P_{t}\chi \right\rangle dt = \left\langle m, \widetilde{\phi}\chi\right\rangle -   \int_0^1\left\langle m, \widetilde{\phi}\chi_t\right\rangle dt,  \\
\end{align*}
where we used the fact that $P_t\chi = \chi - \chi_t$. 
Hence from \eqref{decomposition-of-Y} we have 
\begin{align*}
&\int_{H\times\To^n}Y(w, h)dt\G_h(dw)\lambda(dh)= \left\langle m, Y\right\rangle\\
& = 2 \left\langle m, \chi^2\right\rangle - 2\left\langle m, \chi P_1\chi\right\rangle + 2 \left\langle m, \widetilde{\phi}\chi\right\rangle  - 2 \int_0^1\left\langle m, \widetilde{\phi}\chi_t\right\rangle dt - 2\left\langle m, \chi \chi_1\right\rangle + 2\int_0^1\left\langle m, \widetilde{\phi}\chi_t\right\rangle dt\\
& = 2 \left\langle m, \chi^2\right\rangle - 2\left\langle m, \chi (\chi-\chi_1)\right\rangle + 2 \left\langle m, \widetilde{\phi}\chi\right\rangle  - 2\left\langle m, \chi \chi_1\right\rangle= 2 \left\langle m, \widetilde{\phi}\chi\right\rangle = \s^2
\end{align*}
as in \eqref{eq-the-asymptotic-variance}. The proof is complete. 
\end{proof}

\subsection{Quantitative limit theorems}\label{subsection-rate-of-convergence}

The aim of this subsection is to prove Theorem \ref{rate-of-convergence-SLLN-proof} and Theorem \ref{rate-of-convergence-CLT-proof}. The proofs are given in subsection \ref{subsection-rate-SLLN} and subsection \ref{subsection-rate-CLT} respectively.
We begin with a  result that  gives a convergence rate for the moments of the time average of the observations centered by the quasi-periodic invariant measure. It will be useful when estimating the rate of convergence of the conditioned martingale difference  to the variance. 
\begin{proposition}\label{rate-of-convergence-to-variance}
For any integer $p\geq 1$, $\eta\in (0, \frac{\eta_0}{4p}]$, $\g\in(0,1]$ and $\phi\in C_{\eta, H}^{\g}(H\times\To^n)$, we have 
\begin{align}\label{SLLN-moment-conergence-discrete}
\mathbf{E}_{(w, h)}\left|\frac{1}{N}\sum_{k=1}^N\Big(\phi(X_{k-1})- \big\langle \G_{\b_{k-1}h}, \phi(\cdot, \b_{k-1}h)\big\rangle\Big)\right|^{2p}\leq C_p e^{4p\eta\|w\|^2}\|\phi\|_{\g,\eta, H}^{p}N^{-p}, 
\end{align}
for all $N\geq 1, (w, h)\in H\times\To^n$. The same result also holds  if we replace the summation by integration:
\begin{align}\label{SLLN-moment-conergence-continuous}
\mathbf{E}_{(w, h)}\left|\frac{1}{T}\int_0^T\Big(\phi(X_{t})- \big\langle \G_{\b_{t}h}, \phi(\cdot, \b_{t}h)\big\rangle\Big)dt\right|^{2p}\leq C_p e^{4p\eta\|w\|^2}\|\phi\|_{\g,\eta, H}^pT^{-p},
\end{align}
for any $T\geq 1$ and $(w, h)\in H\times\To^n$. 
\end{proposition}
To show this proposition, we give a lemma first. 
\begin{lemma}\label{a-special-form-multinomial-formula}
For any real numbers $\{x_i\}_{i\geq 1}$ and any integer $m\geq 1, p\geq 1$, let $\ds S_m = \sum_{i=1}^m x_i$ and $S_0=0$. Then one has 
\begin{align*}
\left|S_m\right|^{2p} = \left|\sum_{i=1}^m x_i\right|^{2p} = \sum_{i = 1}^m\sum_{j=i}^mf_{2p-2, j}(x_1, x_2, \cdots, x_{i-1}, x_{i})x_ix_j, 
\end{align*}
where 
\begin{align*}
f_{2p-2, j}(x_1, x_2, \cdots, x_{i-1}, x_{i}) :=\left\{\begin{array}{cc}
\ds\sum_{k=0}^{2p-2}(k+1)S_{i-1}^{k}S_{i}^{2p-2-k}, &\quad \text{when}\, j =i,\\
\ds 2p\sum_{k=0}^{2p-2}S_{i-1}^{k}S_{i}^{2p-2-k}, &\quad \text{when}\, j >i.\\
\end{array}
\right.
\end{align*}
Furthermore,  we have for every $j\geq i$, 
\begin{align}\label{f_{2p-1, j}}
|f_{2p-2, j}|^{\frac{p}{p-1}}\leq C_p(|S_{i-1}|^{2p}+|S_i|^{2p}).
\end{align}

\end{lemma}
\begin{proof}
This follows by recollecting the terms in the multinomial formula: 
\begin{align}\label{multinomial}
\left|\sum_{i=1}^m x_i\right|^{2 p} =  \sum_{k_1+k_2+\cdots+ k_m = 2p}\frac{(2p)!}{k_{1} ! k_{2} ! \cdots k_{m} !}x_1^{k_1}x_2^{k_2}\cdots x_m^{k_m},
\end{align}
where in the summation $0\leq k_1, k_2, \cdots, k_m \leq 2p$ and the convention $x^0=1$ is used. We classify the monomials in the expansion into two categories: 
\begin{enumerate}
\item The exponent of the factor with largest index is 1, like $x_1^{2p-1}x_2, x_1^{p-1}x_2^{p}x_3;$
\item The exponent of the factor with largest index is at least 2, like $x_i^{2p}, x_1^{2p-2}x_3^2$.
\end{enumerate}
Given any monomial, suppose the largest index in its factors is $j$. We rewrite the monomial in the form $C(x_1, x_2, \cdots, x_i)x_ix_j$, where $i=j$ if the monomial belongs to category (2) and $i<j$ with $i$ the second largest index if it belongs to category (1). If we collect those monomials having the same $x_ix_j$, the sum of their corresponding $C$ will be  $f_{2p-2, j}$. 

Let's consider the category (1). In this case, we are collecting monomials in \eqref{multinomial} having the form $x_1^{k_1}x_2^{k_2}\cdots x_i^{k_i}x_ix_j$ with non-negative $k_i$'s summing up to $2p-2$, where $i<j$. Therefore
\begin{align*}
f_{2p-2, j}(x_1, x_2, \cdots, &x_{i-1}, x_{i}) = \sum_{k_1+k_2+\cdots+ k_i = 2p-2}\frac{(2p)!}{k_1!k_2!\cdots k_{i-1}!(k_i+1)!}x_1^{k_1}x_2^{k_2}\cdots x_i^{k_i}\\
&=\sum_{k_i=0}^{2p-2}\frac{(2p)!x_i^{k_i}}{(2p-2-k_i)!(k_i+1)!} \sum_{k_1+k_2+\cdots+k_{i-1}= 2p-2-k_i}\frac{(2p-2-k_i)!}{k_1!k_2!\cdots k_{i-1}!}x_1^{k_1}x_2^{k_2}\cdots x_{i-1}^{k_{i-1}}\\
& = 2p\sum_{k_i = 0}^{2p-2} C_{2p-1}^{k_i+1}S_{i-1}^{2p-2-k_i}x_i^{k_i} = 2px_i^{-1}\left(\sum_{\ell = 0}^{2p-1} C_{2p-1}^{\ell}S_{i-1}^{2p-1-\ell}x_i^{\ell}-S_{i-1}^{2p-1}\right)\\
& = 2p x_i^{-1}\left(S_{i}^{2p-1} - S_{i-1}^{2p-1}\right) =2p\sum_{k=0}^{2p-2}S_{i-1}^{k}S_{i}^{2p-2-k},
\end{align*}
where we used the multinomial formula in the third equality, binomial formula at the penultimate equality, the fact $x_i=S_i-S_{i-1}$ and geometric sum formula at the last step.

 The way to deal with category (2) is similar. Here we are collecting monomials in \eqref{multinomial} having the form $x_1^{k_1}x_2^{k_2}\cdots x_i^{k_i}x_ix_j$ with $i=j$. Hence 
\begin{align*}
f_{2p-2, i}(x_1, &x_2, \cdots, x_{i-1}, x_{i}) = \sum_{k_1+k_2+\cdots+ k_i = 2p-2}\frac{(2p)!}{k_1!k_2!\cdots k_{i-1}!(k_i+2)!}x_1^{k_1}x_2^{k_2}\cdots x_i^{k_i}\\
&=\sum_{k_i=0}^{2p-2} \frac{(2p)!x_i^{k_i}}{(2p-2-k_i)!(k_i+2)!}\sum_{k_1+k_2+\cdots+ k_{i-1} = 2p-2-k_i}\frac{(2p-2-k_i)!}{k_1!k_2!\cdots k_{i-1}!}x_1^{k_1}x_2^{k_2}\cdots x_{i-1}^{k_{i-1}}\\
& = \sum_{k_i = 0}^{2p-2} C_{2p}^{k_i+2}S_{i-1}^{2p-2-k_i}x_i^{k_i} = x_i^{-2}\sum_{k_i = 0}^{2p-2} C_{2p}^{k_i+2}S_{i-1}^{2p-2-k_i}x_i^{k_i+2}\\
& = x_i^{-2}\left(S_{i}^{2p} - S_{i-1}^{2p} - 2pS_{i-1}^{2p-1}x_i\right)=x_i^{-2}\left(S_{i}^{2p} - S_iS_{i-1}^{2p-1}+S_iS_{i-1}^{2p-1}- S_{i-1}^{2p} - 2pS_{i-1}^{2p-1}x_i\right)\\
& =  x_i^{-2}\left(S_{i}\left(S_{i}^{2p-1} - S_{i-1}^{2p-1}\right) - (2p-1)S_{i-1}^{2p-1}x_i\right)\\
& = \sum_{k=0}^{2p-2}(k+1)S_{i-1}^{k}S_{i}^{2p-2-k},
\end{align*}
where in the last equality, we used the fact $x_i=S_i-S_{i-1}$  and geometric sum formula. Note that by Young's inequality for products, we have 
\begin{align*}
\left|S_{i-1}^{k}S_{i}^{2p-2-k}\right|\leq \frac{k}{2p-2}|S_{i-1}|^{2p-2}+\frac{2p-2-k}{2p-2}|S_{i}|^{2p-2}\leq |S_{i-1}|^{2p-2}+|S_{i}|^{2p-2}. 
\end{align*}
Combining this fact with Jensen's inequality for the function $J(x)=x^{\frac{p}{p-1}}$ and the above representation for $f_{2p-2, j}$,  we obtain the desired \eqref{f_{2p-1, j}}. 
\end{proof}
\begin{proof}[Proof of Proposition \ref{rate-of-convergence-to-variance}]
We divide the proof into two steps. The general idea is to use the exponential decay from mixing to compensate the growth in the sum. 

\vskip0.05in

{\it Step 1: Proof of discrete case \eqref{SLLN-moment-conergence-discrete}.} Recall that $\widetilde{\phi}(w, h) = \phi(w, h) - \langle\G_h, \phi(\cdot, h)\rangle$, then the summands in inequality \eqref{SLLN-moment-conergence-discrete} is $\xi_i: = \widetilde{\phi}(X_{i-1})$. Let  
\[S_m =\sum_{i=1}^{m}\xi_i, \quad s_N = \sup_{1\leq m\leq N}\mathbf{E}_{(w, h)}|S_m|^{2p}.\]
Let $g(\xi_i, \xi_j) = \xi_i\mathbf{E}_{(w, h)}\left[\xi_j|\mathcal{F}_{i-1}\right] $ and $g_p(w, h)= \left(\mathbf{E}_{(w, h)}\left|g(\xi_i, \xi_j)\right|^p\right)^{1/p}$.  

\vskip0.05in

Note that by Theorem \ref{theorem-mixing-erighted-observables},  the sequence $\xi_i$ of random variables  decays exponentially in expectation. The monomials in the expansion of $|S_m|^{2p}$ are the product of such random variables.  Lemma \ref{a-special-form-multinomial-formula} and Markov property enable us to decompose the monomials into the product of sums with factors having largest time index.  Then the exponentially decaying factors will compensate the growth of the sum, which leads to a desired slow growth of $|S_N|^{2p}$ in a order of $N^p$. Indeed, it follows from Lemma \ref{a-special-form-multinomial-formula} and the Hölder inequality that 
\begin{align*}
\mathbf{E}_{(w, h)}|S_m|^{2p} &=\mathbf{E}_{(w, h)}\sum_{i = 1}^m\sum_{j=i}^mf_{2p-2, j}(\xi_1, \xi_2, \cdots, \xi_{i-1}, \xi_{i})\xi_i\xi_j\\
& =\sum_{i = 1}^m\sum_{j=i}^m \mathbf{E}_{(w, h)}\Big[f_{2p-2, j}(\xi_1, \xi_2, \cdots, \xi_{i-1}, \xi_{i})\xi_i\mathbf{E}_{(w, h)}\left[\xi_j|\mathcal{F}_{i-1}\right]\Big]\\
&\leq \sum_{i = 1}^m\sum_{j=i}^m \left(\mathbf{E}_{(w, h)}\left|f_{2p-2, j}(\xi_1, \xi_2, \cdots, \xi_{i-1}, \xi_{i})\right|^{\frac{p}{p-1}}\right)^{\frac{p-1}{p}}g_p(w, h)\\
&\leq C_p\sum_{i = 1}^m\sum_{j=i}^m\Big(\mathbf{E}_{(w, h)}\left[|S_{i-1}|^{2p}+|S_{i}|^{2p}\right]\Big)^{\frac{p-1}{p}}g_p(w, h).
\end{align*}
Taking supremum for $1\leq m\leq N$, one has 
\begin{align*}
s_N\leq C_p s_N^{\frac{p-1}{p}}\sum_{i = 1}^N\sum_{j=i}^N \left(\mathbf{E}_{(w, h)}\left|g(\xi_i, \xi_j)\right|^p\right)^{1/p}. 
\end{align*}
Hence 
\begin{align*}
s_N\leq C_p \left(\sum_{i = 1}^N\sum_{j=i}^N \left(\mathbf{E}_{(w, h)}\left|g(\xi_i, \xi_j)\right|^p\right)^{1/p}\right)^p.
\end{align*}
Note that by Theorem \ref{theorem-mixing-erighted-observables} with $\k=2$, 
\begin{align*}
\mathbf{E}_{(w, h)}\left|g(\xi_i, \xi_j)\right|^p &= \mathbf{E}_{(w, h)}\left|\widetilde{\phi}(X_{i-1})\mathbf{E}_{(w, h)}\left[\widetilde{\phi}(X_{j-1})|\mathcal{F}_{i-1}\right]\right|^p \\
&\leq  \mathbf{E}_{(w, h)}\left|\widetilde{\phi}(X_{i-1})P_{j-i}\widetilde{\phi}(X_{i-1})\right|^p\\
&\leq C^p\|\phi\|_{\g,\eta, H}^pe^{-p\L(j-i)} \mathbf{E}_{(w, h)}e^{4p\eta\|\Phi_{0, i-1, h}(w)\|^2}\\
&\leq C^p\|\phi\|_{\g,\eta, H}^pe^{-p\L(j-i)}e^{4p\eta\|w\|^2},
\end{align*}
for $\eta\in (0, \frac{\eta_0}{4p}]$ by estimate \eqref{eq: est1}. Therefore 
\begin{align*}
s_N\leq C_p \|\phi\|_{\g,\eta, H}^{p}e^{4p\eta\|w\|^2}N^p,
\end{align*}
Dividing both sides of the above inequality by $N^{2p}$ completes the proof of the first estimate \eqref{SLLN-moment-conergence-discrete} in Proposition \ref{rate-of-convergence-to-variance}. 

\vskip0.05in

{\it Step 2: Proof of continuous case \eqref{SLLN-moment-conergence-continuous}.}
This follows from a similar argument as in the previous step. Let $\xi(t) = \phi(X_{t})- \big\langle \G_{\b_{t}h}, \phi(\cdot, \b_{t}h)\big\rangle$, $I_{r} = \int_0^{r}\xi(t)dt $ and $\ds\mathcal{I}_{T} =\sup_{0\leq r\leq T}\mathbf{E}_{(w, h)}\left|I_r\right|^{2p} $. We first note that 
\begin{align*}
I_{r}^{2p} &= \int_{[0, r]^{2p}}\xi(t_1)\xi(t_2)\cdots\xi(t_{2p})dt_1dt_2\cdots dt_{2p}\\
& = (2p)!\int_{0\leq t_1\leq t_2\leq\cdots\leq t_{2p}\leq r} \xi(t_1)\xi(t_2)\cdots\xi(t_{2p})dt_1dt_2\cdots dt_{2p}.
\end{align*}
For $r_1\leq r_2$, denote $\varphi(r_1, r_2) = \xi(r_1)\mathbf{E}_{(w, h)}[\xi(r_2)|\mathcal{F}_{r_1}]$. Then 
\begin{align*}
&\mathbf{E}_{(w, h)}\left|I_r\right|^{2p} = (2p)!\mathbf{E}_{(w, h)}\int_{0\leq t_1\leq t_2\leq\cdots\leq t_{2p}\leq r} \xi(t_1)\xi(t_2)\cdots\xi(t_{2p-2})\varphi(t_{2p-1}, t_{2p})dt_1dt_2\cdots dt_{2p}\\
& =(2p)!\mathbf{E}_{(w, h)}\left( \int_{0}^{r}\int_{0}^{t_{2p}}\varphi(t_{2p-1}, t_{2p})\left(\int_{0\leq t_1\leq\cdots\leq t_{2p-1}}\xi(t_1)\cdots\xi(t_{2p-2})dt_1\cdots dt_{2p-1}\right)dt_{2p-1}dt_{2p}\right)\\
&\leq 2p(2p-1)\int_{0}^{r}\int_{0}^{t_{2p}}\left(\mathbf{E}_{(w, h)}\left|\varphi(t_{2p-1}, t_{2p})\right|^{p}\right)^{\frac1p} \left(\mathbf{E}_{(w, h)} \left|I_{t_{2p-1}}\right|^{2p}\right)^{\frac{p-1}{p}}dt_{2p-1}dt_{2p}.
\end{align*}
Taking the supremum w.r.t $r$ over $[0, T]$, we have 
\begin{align*}
\mathcal{I}_{T}  \leq 2p(2p-1) \left(\mathcal{I}_{T}\right)^{\frac{p-1}{p}}\int_{0}^{T}\int_0^{t_2}\left(\mathbf{E}_{(w, h)}\left|\varphi(t_{1}, t_{2})\right|^{p}\right)^{\frac1p} dt_1dt_2. 
\end{align*}
Like in the proof of \eqref{SLLN-moment-conergence-discrete}, one has for $\eta\in (0, \frac{\eta_0}{4p})$, 
\[\left(\mathbf{E}_{(w, h)}\left|\varphi(t_{1}, t_{2})\right|^{p}\right)^{\frac{1}{p}}\leq C\|\phi\|_{\g,\eta, H}e^{-\L(t_2-t_1)}e^{4\eta\|w\|^2}. \]
Therefore
\begin{align*}
\mathcal{I}_{T} \leq C_p\|\phi\|_{\g,\eta, H}^p T^pe^{4p\eta\|w\|^2}
\end{align*}
and the inequality \eqref{SLLN-moment-conergence-continuous} follows by dividing both sides with $T^{2p}$.
\end{proof}
\vskip0.05in 

\subsubsection{SLLN with convergence rate}\label{subsection-rate-SLLN} This is a consequence of \eqref{SLLN-moment-conergence-continuous}, the error estimate in Lemma \ref{the-error-term} and the Borel-Cantelli lemma. We now give the details. 
\begin{proof}[Proof of Theorem \ref{rate-of-convergence-SLLN-proof}]
For any $\varepsilon>0$, let $E_N =\left \{\o\in\O: \left|\frac{1}{N}M_N\right|>N^{-(1/2-\varepsilon)}\right\}$. From Proposition \ref{stepwise-centered-corrector-mapping}, we know that for $\eta\in(0, \eta_0/2]$, $\chi\in C^{\g_0}_{2\eta, H}(H\times\To^n)$. Thus $\chi^{2^p}\in C^{\g_0}_{2^{p+1}\eta, H}(H\times\To^n)$. Hence from estimate \eqref{SLLN-moment-conergence-continuous} and \eqref{eq: est1} and Markov's inequality, we have for $\eta\in(0, 2^{-(p+1)}\eta_0]$
\begin{align*}
&\mathbf{P}(E_N) \leq N^{2^p(-1/2-\varepsilon)}\mathbf{E}_{(w, h)}M_N^{2^p}\\
&\leq C_pN^{2^p(-1/2-\varepsilon)}\mathbf{E}_{(w, h)}\left(\chi(X_{N})^{2^p}+ \chi(w, h)^{2^p} + \left(\int_0^{N}\widetilde{\phi}(X_t)dt\right)^{2^p}\right)\\
&\leq C_pN^{2^p(-1/2-\varepsilon)}\left( \|\chi\|_{\g_0, 2^{p+1}\eta, H}\mathbf{E}e^{2^{p+1}\eta\|\Phi_{0, N, h}(w)\|^2} + \|\chi\|_{\g_0, 2^{p+1}\eta, H}e^{2^{p+1}\eta\|w\|^2} + e^{2^{p+1}\eta\|w\|^2}\|\phi\|_{\g,\eta, H}^{2^{p-1}}N^{2^{p-1}}\right)\\
&\leq  C_p(\|\phi\|_{\g,\eta, H})N^{-2^p\varepsilon}e^{2^{p+1}\eta\|w\|^2}.
\end{align*}
For any $\varepsilon>0$, and every integer $p$ such that $2^{p}\varepsilon>1$, 
\[\sum_{N=1}^{\infty}\mathbf{P}(E_N)<\infty.\]
By the Borel-Cantelli lemma, there is an almost surely finite random time $N_1(\o)$, such that for all $N>N_1(\o)$, 
\[\left|\frac{1}{N}M_N\right|\leq N^{-(1/2-\varepsilon)}.\]
Note that for $\ell>0$, 
\begin{align*}
\mathbf{E}N_1^{\ell} = \sum_{k=1}^{\infty}\mathbf{P}(N_1 = k) k ^{\ell}&\leq \sum_{k=1}^{\infty}\mathbf{P}(E_k)
k^{\ell}\\
&\leq \sum_{k=1}^{\infty}C_p(\|\phi\|_{\g,\eta, H})k^{\ell-2^p\varepsilon}e^{2^{p+1}\eta\|w\|^2}\leq C_p(\|\phi\|_{\g,\eta, H},\ell, \varepsilon)e^{2^{p+1}\eta\|w\|^2}
\end{align*}
as long as $\ell<2^p\varepsilon-1$. In a similar fashion we can estimate the moments of the random time $N_0(\o)$ in \eqref{convergence-rate-error}. Let $\ell>0$, then for $\eta\in(0, 2^{-p-1}\eta_0]$, 
\begin{align*}
\mathbf{E}N_0^{\ell} &= \sum_{k=1}^{\infty}\mathbf{P}(N_0 = k) k ^{\ell}\\
&\leq  \sum_{k=1}^{\infty}\mathbf{P}\left(\sup_{k\leq t\leq k+1}R_{k, t}> k^{1/4}\right)k ^{\ell}\leq  \sum_{k=1}^{\infty}\mathbf{E}\left(\sup_{k\leq t\leq k+1}R_{k, t}^{2^p}\right)k^{\ell-2^{p-2}}\\
&\leq C_p(\|\phi\|_{\g,\eta,H})\sum_{k=1}^{\infty}\mathbf{E}\sup_{k\leq t\leq k+1}\exp(2^{p+1}\eta\|\Phi_{0, t, h}\|^2)k^{\ell-2^{p-2}}\\
&\leq  C_p(\|\phi\|_{\g,\eta,H})\sum_{k=1}^{\infty}e^{2^{p+1}\eta\|w\|^2}k^{\ell-2^{p-2}}=C_p(\|\phi\|_{\g,\eta,H}, \ell)e^{2^{p+1}\eta\|w\|^2},
\end{align*}
provided that $\ell<2^{p-2}-1$. The conclusion of Theorem \ref{rate-of-convergence-SLLN-proof} then follows from the above estimates, the martingale approximation \eqref{discrete-martingale-approximation} and Lemma \ref{the-error-term}. 
\end{proof}

\subsubsection{CLT with convergence rate}\label{subsection-rate-CLT}
We recall  a Berry-Esseen type estimate for martingales from \cite{HH14}, which will be used in the proof of the central limit theorem. 
\begin{theorem}[Theorem 3.10 of \cite{HH14}]\label{theorem-Berry-Esseen}
Let $\ds M_N=\sum_{j=1}^NZ_j$ be a zero mean martingale and $\ds\s_k^2 = \sum_{j=1}^k\mathbf{E}Z_j^2$. If  $q>\frac12$, and 
\begin{align}\label{Berry-Esseen-condition1}
\max_{j\leq N}\frac{1}{\s_N^{4q}}\mathbf{E}|Z_j|^{4q}\leq \frac{M}{N^{2q}},
\end{align}
for a constant $M>0$. Then there exists a constant $C$ depending only on $M$ and $q$ such that whenever 
\begin{align}\label{Berry-Esseen-condition2}
N^{-q} + \mathbf{E}\left|\frac{1}{\s_N^2}\sum_{j=1}^N\mathbf{E}\left[Z_j^2|\mathcal{F}_{j-1}\right] -1\right|^{2q}\leq 1,
\end{align}
one has
\begin{align}\label{Berry-Esseen-Conslusion}
\sup_{z\in\R}\left|\mathbf{P}\left(\frac{M_N}{\s_N}\leq z\right) - \N(z)\right|\leq C\left(N^{-q} + \mathbf{E}\left|\frac{1}{\s_N^2}\sum_{j=1}^N\mathbf{E}\left[Z_j^2|\mathcal{F}_{j-1}\right]-1\right|^{2q}\right)^{1/(4q+1)},
\end{align}
where $\N(z)$ is the distribution function of the standard normal distribution. 
\end{theorem}
We are now in a position to prove the central limit theorem with convergence rate. 
\begin{proof}[Proof of Theorem \ref{rate-of-convergence-CLT-proof}]
We first show that the average for martingale difference square $\frac{1}{N}\s_N^2$ converges to the asymptotic variance $\s^2$. Based on this, we then prove the convergence rate of CLT for the approximating martingale sequence. This is divided into two cases according to if $\s^2=0$ or not. In the case $\s^2\neq 0$, we prove the convergence rate by Berry-Esseen theorem through verifying  the conditions \eqref{Berry-Esseen-condition1} and \eqref{Berry-Esseen-condition2}, and a specific estimate of the expectation on the right hand side of the conclusion \eqref{Berry-Esseen-Conslusion}. In the case $\s^2=0$, we only need the convergence rate of $\frac{1}{N}\s_N^2$ to $\s^2$. Then we pass to the desired convergence rate in Theorem \ref{rate-of-convergence-CLT-proof} through the martingale approximation \eqref{discrete-martingale-approximation}. The proof is divided into three steps. 

\vskip0.05in 

Recall from \eqref{a-formula-for-sigma^2} that the asymptotic variance 
\begin{align}\label{Recall-asymptotic-variance}
\s^2 = \int_{H\times\To^n}Y(w,h)\G_{h}(dw)\lambda(dh) = \int_{\To^n}F(h)\lambda(dh),
\end{align}
where $F(h) = \langle\G_h, Y(\cdot, h)\rangle$.  Also by the Markov property of the homogenized process, we have 
\begin{align}\label{Markov-Martingale-Difference}
\mathbf{E}\left[Z_j^2|\mathcal{F}_{j-1}\right] = Y(X_{j-1}), \quad j\geq 1.
\end{align}
\vskip0.05in 

{\it Step 1: Convergence of $\frac{1}{N}\s_N^2$ to $\s^2$.} The convergence follows from moments convergence in \eqref{SLLN-moment-conergence-discrete} with observable function $Y$, and the convergence of the Birkhoff average of the irrational rotational flow on the time symbol space $\To^n$ with observable function $F$ in \eqref{Recall-asymptotic-variance}. The later convergence is the place where the Diophantine condition \eqref{eq-Diophantine condition} comes into play. 

\vskip0.05in 

To be specific, note that $Y \in C^{\overline{\g}}_{2^5\eta, H}(H\times\To^n)$ by Proposition \ref{Holder-regularity-of-Y}, thus Proposition \ref{rate-of-convergence-to-variance} implies that for $\eta\in(0, 2^{-7}p^{-1}\eta_0]$, and $N\geq 1$, 
\begin{align}\label{rate-of-convergence-to-variance-1}
\mathbf{E}_{(w, h)}\left|\frac{1}{N}\sum_{j=1}^N\Big(Y(X_{j-1})- F(\b_{j-1}h)\Big)\right|^{2p}\leq C_p e^{2^{7}p\eta\|w\|^2}\|Y\|_{\overline{\g}, 2^5\eta, H}^pN^{-p}. 
\end{align}
In addition, since $F\in C^{\overline{\g}_0}(\To^n, \R)$ by Proposition \ref{Holder-regularity-of-Y}, it follows from Theorem 3 in \cite{KLM19} that for $N\geq 1$, 
\begin{align}\label{convergence-rate-irrtational-rotation}
\left|\frac{1}{N}\sum_{j=1}^NF(\b_{j-1}h) - \s^2\right|\leq C\|F\|_{\overline{\g}_0}N^{-\frac{\overline{\g}_0}{A+n}}, 
\end{align}
where $A$ is the constant from the Diophantine condition \eqref{eq-Diophantine condition} and $n$ is the dimension of the torus. Therefore, by \eqref{Markov-Martingale-Difference}, triangle inequality and inequality \eqref{convergence-rate-irrtational-rotation}, it follows that for $\eta\in(0,2^{-7}\eta_0]$, and $N\geq 1$, 
\begin{align}\label{mean-convergence-rate-to-variance}
\begin{split}
&\left|\frac{1}{N}\s_N^2-\s^2\right| = \left|\frac{1}{N}\sum_{j=1}^N\mathbf{E}_{(w, h)}Z_{j}^2-\s^2\right|\\
&\leq \Bigg|\frac{1}{N}\sum_{j=1}^N\Big(\mathbf{E}_{(w, h)}\left[\mathbf{E}_{(w, h)}\left[Z_{j}^2|\mathcal{F}_{j-1}\right]\right] - F(\b_{j-1}h)\Big)\Bigg|+ \left|\frac{1}{N}\sum_{j=1}^NF(\b_{j-1}h)-\s^2\right|\\
&\leq \mathbf{E}_{(w, h)}\left|\frac{1}{N}\sum_{j=1}^N\Big(Y(X_{j-1}) - \big\langle \G_{\b_{j-1}h}, Y(\cdot, \b_{j-1}h)\big\rangle\Big)\right| +C\|F\|_{\overline{\g}_0}N^{-\frac{\overline{\g}_0}{A+n}}\\
&\leq Ce^{2^{6}\eta\|w\|^2}\|Y\|_{\overline{\g}, 2^5\eta, H}^{1/2}N^{-1/2}+ C\|F\|_{\overline{\g}_0}N^{-\frac{\overline{\g}_0}{A+n}},
\end{split}
\end{align}
where we used H\"older inequality and  inequality \eqref{rate-of-convergence-to-variance-1} with $p=1$ in the last step for the expectation. This shows that $\frac{1}{N}\s_N^2$ converges to $\s^2$. 

\vskip0.05in

{\it Step 2: Convergence rate of CLT for the approximating martingale sequence.} This is divided into two cases. For $\s^2\neq 0$, we apply the Berry-Esseen theorem. For $\s^2=0$ we will use \eqref{mean-convergence-rate-to-variance}. 

\vskip0.05in

{\it Case $\s^2\neq 0$}: Since $\frac{1}{N}\s_N^2$ converges to $\s^2$ from the previous step, we see that $\frac{1}{N}\s_N^2\geq \s^2/2$ for large $N$. 
This fact, together with $F$ as in \eqref{Recall-asymptotic-variance}, equality \eqref{Markov-Martingale-Difference}, and triangle inequality,  imply that in \eqref{Berry-Esseen-condition2}, 
\begin{align}\label{Verify-Berry-Essen2-1}
\begin{split}
&\left|\frac{1}{\s_N^2}\sum_{j=1}^N\mathbf{E}\left[Z_j^2|\mathcal{F}_{j-1}\right] -1\right| = \frac{N}{\s_N^2}\left|\frac{1}{N}\sum_{j=1}^NY(X_{j-1})-\frac{\s_N^2}{N}\right|\\
&\leq \frac{2}{\s^2}\left(\left|\frac{1}{N}\sum_{j=1}^N\Big(Y(X_{j-1}) - F(\b_{j-1}h)\Big)\right| + \left|\frac{1}{N}\sum_{j=1}^NF(\b_{j-1}h)-\s^2\right|+\left|\s^2-\frac{\s_N^2}{N}\right|\right).
\end{split}
\end{align}
Set $q= 2^{p-2}$ for integer $p\geq2$ in Theorem \ref{theorem-Berry-Esseen}.  It follows from Lemma \eqref{martingalebounds}, and $\frac{1}{N}\s_N^2\geq \s^2/2$ that the condition \eqref{Berry-Esseen-condition1} holds for $p\geq 2$ and $\eta\in (0, 2^{-p-1}\eta_0]$. 

\vskip0.05in

Noting that inequality \eqref{rate-of-convergence-to-variance-1} with $p$ replaced by $2^{p-2}$ implies that for $\eta\in (0, 2^{-p-5}\eta_0]$, one has 
\begin{align}\label{rate-of-convergence-to-variance-1-again}
\mathbf{E}_{(w, h)}\left|\frac{1}{N}\sum_{j=1}^N\Big(Y(X_{j-1})- \big\langle \G_{\b_{j-1}h}, Y(\cdot, \b_{j-1}h)\big\rangle\Big)\right|^{2^{p-1}}\leq C_p e^{2^{p+5}\eta\|w\|^2}\|Y\|_{\overline{\g}, 2^5\eta, H}^pN^{-2^{p-2}}. 
\end{align}
Applying Jensen's inequality for the function $J(x) = x^{2^{p-1}}$ to estimate \eqref{Verify-Berry-Essen2-1}, then taking expectation and using estimates \eqref{mean-convergence-rate-to-variance}, \eqref{rate-of-convergence-to-variance-1-again} and \eqref{convergence-rate-irrtational-rotation}, we obtain 
\begin{align*}
&\mathbf{E}_{(w, h)}\left|\frac{1}{\s_N^2}\sum_{j=1}^N\mathbf{E}_{(w, h)}\left[Z_j^2|\mathcal{F}_{j-1}\right]-1\right|^{2^{p-1}}\\
&\leq C_p\left(e^{2^{p+5}\eta\|w\|^2}\|Y\|_{\overline{\g}, 2^5\eta, H}^{2^{p-2}}N^{-2^{p-2}} + \left|\frac{1}{N}\sum_{j=1}^NF(\b_{j-1}h) -\s^2\right|^{2^{p-1}}+\left|\frac{1}{N}\s_{N}^2-\s^2\right|^{2^{p-1}}\right)\\
&\leq C_p\left(e^{2^{p+5}\eta\|w\|^2}\|Y\|_{\overline{\g}, 2^5\eta, H}^{2^{p-2}}N^{-2^{p-2}} +\|F\|_{\overline{\g}_0}^{2^{p-1}}N^{-\frac{2^{p-1}\overline{\g}_0}{A+n}}\right).
\end{align*}
Then condition \eqref{Berry-Esseen-condition2} is satisfied for large $N$. 

\vskip0.05in

Therefore by Theorem \ref{theorem-Berry-Esseen} we have the existence of a large $K$ such that for all $N\geq K$
\begin{align}\label{CLT-Martingale-01}
\begin{split}
\sup_{z\in\R}\left|\mathbf{P}\left(\frac{M_N}{\s_N}\leq z\right) - \N(z)\right|&\leq C\left(N^{-2^{p-2}} + \mathbf{E}\left|\frac{1}{\s_N^2}\sum_{j=1}^N\mathbf{E}\left[Z_j^2|\mathcal{F}_{j-1}\right]-1\right|^{2^{p-1}}\right)^{1/(2^p+1)}\\
&\leq Ce^{2^{5}\eta\|w\|^2}\left(N^{-\frac{2^{p-2}}{2^p+1}} +N^{-\frac{2^{p-1}\overline{\g}_0}{(2^p+1)(A+n)}} \right).
\end{split}
\end{align}
As a result,
\begin{align}\label{CLT-Martingale-s>0}
\begin{split}
\sup_{z\in\R}\left|\mathbf{P}\left(\frac{M_N}{\sqrt{N}}\leq z\right) - \N_{\s}(z)\right|&\leq \sup_{z\in\R}\left|\mathbf{P}\left(\frac{M_N}{\sqrt{N}}\leq z\right) - \N\left(\frac{\sqrt{N}}{\s_N}z\right)\right| + \sup_{z\in\R}\left|\N\left(\frac{\sqrt{N}}{\s_N}z\right) - \N_{\s}(z)\right|\\
&\leq \sup_{z\in\R}\left|\mathbf{P}\left(\frac{M_N}{\s_N}\leq z\right) - \N\left(z\right)\right| +C\left|\frac{\s_N}{\sqrt{N}}-\s\right|\\
&\leq Ce^{2^{5}\eta\|w\|^2}\left(N^{-\frac{2^{p-2}}{2^p+1}} +N^{-\frac{2^{p-1}\overline{\g}_0}{(2^p+1)(A+n)}} \right).
\end{split}
\end{align}
 by Lipschitz continuity of normal distribution function and the fact $\N_{\s}(z)=\N\left(\frac{z}{\s}\right)$, together with \eqref{CLT-Martingale-01}. 
 
\vskip0.05in

{\it Case $\s^2=0$}: Note that for $z<0$, $\N_{0}(z)=0$, hence 
\[\left|\mathbf{P}\left(\frac{M_N}{\sqrt{N}}\leq z\right) - \N_0(z)\right| \leq  \mathbf{P}\left(\left|\frac{M_N}{\sqrt{N}}\right|\geq |z|\right) .\]
For $z\geq0$, $\N_0(z)=1$, so 
\[\left|\mathbf{P}\left(\frac{M_N}{\sqrt{N}}\leq z\right) - \N_0(z)\right|=\mathbf{P}\left(\frac{M_N}{\sqrt{N}}>z\right)\leq   \mathbf{P}\left(\left|\frac{M_N}{\sqrt{N}}\right|\geq |z|\right) .\]
Recall from the martingale approximation \eqref{discrete-martingale-approximation} that the martingale difference $Z_k= M_k-M_{k-1}$ for $k\geq 1$ with $M_0=0$, hence 
\[\mathbf{E}_{(w, h)}|M_N|^2 =\mathbf{E}_{(w, h)}\left(\sum_{k=1}^NZ_k\right)^2 =\sum_{k=1}^N\mathbf{E}_{(w, h)}Z_k^2.\]
Combining these facts and Markov inequality we obtain 
\begin{align}\label{CLT-Martingale-s=0}
\begin{split}
&(|z|\wedge 1)\left|\mathbf{P}\left(\frac{M_N}{\sqrt{N}}\leq z\right) - \N_0(z)\right|\leq (|z|\wedge 1)\left|\mathbf{P}\left(\left|\frac{M_N}{\sqrt{N}}\right|\geq |z|\right)\right|\\
&\leq  (|z|\wedge 1)|z|^{-1}N^{-1/2}\mathbf{E}_{(w, h)}|M_N|\leq N^{-1/2}\left(\mathbf{E}_{(w, h)}|M_N|^2\right)^{1/2}= \left(\frac{1}{N}\sum_{k=1}^N\mathbf{E}_{(w, h)}Z_{k}^2\right)^{1/2}\\
&\leq C e^{2^{5}\eta\|w\|^2}\left(N^{-1/4}+ N^{-\frac{\overline{\g}_0}{2(A+n)}}\right)
\end{split}
\end{align}
by estimate \eqref{mean-convergence-rate-to-variance}.

\vskip0.05in

{\it Step 3: Pass to the desired estimates in Theorem \ref{rate-of-convergence-CLT-proof}.}
To prove the estimates for the solution process through the estimates for martingale sequence in the previous step, we apply the martingale approximation \eqref{discrete-martingale-approximation} and the following lemma from \cite{Shi06}. 
\begin{lemma}\label{lemma-pass-to-continuous}
Let $R_1, R_2$ be real random variables. Then for any $\s\geq 0$ and $\varepsilon>0$ we have 
\[\sup_{z\in\R}|\D_{\s}(R_1,z)|\leq \sup_{z\in\R}|\D_{\s}(R_2, z)| + \mathbf{P}(|R_1-R_2|>\varepsilon)+c_{\s}\varepsilon, \]
where $c_{\s}$ is a constant depending only on $\s$ and $\D_{\s}(R, z)$ for random variable $R$ is defined as 
\begin{align*}
\D_{\s}(R, z) :=\left\{\begin{array}{lr}
 \mathbf{P}(R\leq z) - \N_{\s}(z)\text{ , } &\s> 0,\\
(|z|\wedge 1)\left(\mathbf{P}(R\leq z) - \N_0(z)\right)\text{ , }& \s = 0.\\
\end{array}
\right.
\end{align*}
\end{lemma}
Recall that $N$ is the integer part of $T$ in the martingale approximation \eqref{discrete-martingale-approximation}, hence
\[\left|\frac{1}{\sqrt{T}} - \frac{1}{\sqrt{N}}\right|\leq\frac{1}{\sqrt{NT}}.\]
 It then follows from the approximation that 
\begin{align*}
\left|\frac{1}{\sqrt{T}}\int_0^T\widetilde{\phi}(X_t) dt - \frac{M_N}{\sqrt{N}}\right| \leq \frac{1}{\sqrt{NT}}\left|\int_0^T\widetilde{\phi}(X_t) dt\right| + \frac{R_{N,T}}{\sqrt{N}}.
\end{align*}
The estimates of  the remainder term $R_{N,T}$ obtained in the proof of Lemma \ref{the-error-term} together with the Markov inequality yields 
\begin{align*}
\mathbf{P}\left(\frac{R_{N,T}}{\sqrt{N}}>\varepsilon/2\right)\leq Ce^{2\eta\|w\|^2}N^{-4}\varepsilon^{-8}.
\end{align*}

It follows from \eqref{SLLN-moment-conergence-continuous} with $p=1$ that 
\begin{align*}
\mathbf{E}\left|\int_0^T\widetilde{\phi}(X_t) dt\right|\leq\left(\mathbf{E}\left(\int_0^T\widetilde{\phi}(X_t) dt\right)^2\right)^{\frac12}\leq T^{\frac12} Ce^{2\eta\|w\|^2}.
\end{align*}
Now applying Lemma \ref{lemma-pass-to-continuous} with $R_1 = \frac{1}{\sqrt{T}}\int_0^T\widetilde{\phi}(X_t) dt $ and $R_2 = \frac{M_N}{\sqrt{N}}$, and the Markov inequality we have 
\begin{align*}
\sup_{z\in\R}|\D_{\s}(R_1,z)|
&\leq \sup_{z\in\R}|\D_{\s}(R_2, z)| + \mathbf{P}(|R_1-R_2|>\varepsilon)+c_{\s}\varepsilon\\
&\leq \sup_{z\in\R}|\D_{\s}(R_2, z)| +2 \varepsilon^{-1}N^{-\frac12}\mathbf{E}\left|T^{-\frac12}\int_0^T\widetilde{\phi}(X_t) dt\right| + Ce^{2\eta\|w\|^2}N^{-4}\varepsilon^{-8}+c_{\s}\varepsilon\\
&\leq \sup_{z\in\R}|\D_{\s}(R_2, z)| +Ce^{2\eta\|w\|^2}N^{-\frac14}
\end{align*}
by taking $\varepsilon = N^{-\frac14}$. Combining this with \eqref{CLT-Martingale-s>0} and \eqref{CLT-Martingale-s=0} completes the proof of Theorem \ref{rate-of-convergence-CLT-proof}. 
\end{proof}

\section{Appendix: estimates of the solution}

\vskip0.2in

Several standard estimates about the solution  $\Phi_{0, t, h}(w)$ of the stochastic Navier-Stokes equation \eqref{NS} with time symbol $h\in\To^n$ are collected in the following Lemma \ref{bounds}. Note that for any $h\in\To^n$, we have 
\[\sup_{t\in\R}\|\Psi(\b_t h)\|=\sup_{t\in\R}\|f(t)\|: =\|f\|_{\infty}, \]
therefore the constants in bounds on the solution do not depend on $h$. 
Let $\mathcal{B}_{0}:=\sum_{i=1}^d\|g_i\|^2$ be the energy input from the noise.

\begin{lemma}\label{bounds}
There is a constant $\eta_0=\eta_0(\|f\|_{\infty}, \mathcal{B}_{0}, \nu)>0$, such that for $\eta\in (0, \eta_0]$, $(w, h)\in H\times\To^n$ and $t, \tau\geq 0$, the following estimates hold with constants $C$ and $r_0$ independent of $(w, h)$. 
\begin{equation}\label{eq: est1}
\mathbf{E} \exp \left(\eta\left\|\Phi_{0, t, h}\right\|^{2}\right) \leq C \exp \left(\eta e^{-\nu t}\left\|w\right\|^{2}\right),
\end{equation}
\begin{align}\label{enstrophy}
\quad \mathbf{E} \exp \left(\eta \sup _{t \geq \tau}\left(\left\|\Phi_{0, t, h}\right\|^{2}+\nu \int_{\tau}^{t}\left\|\Phi_{0, r, h}\right\|_{1}^{2} d r-C(t-\tau)\right)\right)\leq C \exp \left(\eta e^{-\nu \tau}\left\|w\right\|^{2}\right),
\end{align}
\begin{align}\label{continuousonhull}
\mathbf{E}\left\|\Phi_{0, t, h_1} - \Phi_{0, t, h_2}\right\|^2\leq Ce^{r_0t}\exp\left({\eta\|w\|^2}\right)\sup_{t\in\R}\|\Psi(\b_th_1)-\Psi(\b_th_2)\|^2,\, h_1,h_2\in \To^n,
\end{align}
\begin{align}\label{continuous-initial-condition}
\mathbf{E}\|\Phi_{0, t, h}(w_1)-\Phi_{0, t, h}(w_2)\|^2\leq C\|w_1 - w_2\|^{2}e^{r_0t}\left(e^{\eta\|w_1\|^2}+e^{\eta\|w_2\|^2}\right),\, w_1, w_2\in H,
\end{align}
\begin{align}\label{higher-regularity}
\mathbf{E}\|\Phi_{0, t, h}(w)\|_1^2\leq C(t)\exp\left(\eta\|w\|^2\right). 
\end{align}
\end{lemma}
Estimates \eqref{eq: est1} and \eqref{enstrophy} follow from \cite{HM06}. We refer \cite{KS12} for a proof of \eqref{higher-regularity}. The regular dependence \eqref{continuousonhull} and \eqref{continuous-initial-condition}  can be proved as in \cite{HM08}. 

\vskip0.2in
\noindent\emph{Acknowledgements.} This work is supported by NSFC (No. 12090010). We sincerely thank all the reviewers for carefully reading and checking the details of this paper. 

\vskip0.5in
\noindent {\bf Data Availability} \quad No data has been produced in the original research reported in this manuscript.

\vskip0.2in
\noindent {\bf Competing Interests} \quad The authors have no competing interests to declare that are
relevant to the content of this article.


\end{document}